\documentclass[onefignum,onetabnum]{siamart251216}

\usepackage{lipsum,mathtools}
\usepackage{amsfonts}
\usepackage{graphicx}
\usepackage{epstopdf}
\usepackage{algpseudocode}
\usepackage{subcaption}
\usepackage{csquotes}
\usepackage{algorithm}
\ifpdf
  \DeclareGraphicsExtensions{.eps,.pdf,.png,.jpg}
\else
  \DeclareGraphicsExtensions{.eps}
\fi

\newcommand{\R}{\ensuremath{\mathbb{R}}}
\newcommand{\N}{\ensuremath{\mathbb{N}}}
 \usepackage{bm}
\usepackage[protrusion=true,expansion=false]{microtype}
 \newcommand{\E}{\mathbb{E}}
 \newcommand{\norm}[1]{\left\lVert#1\right\rVert}
 \newcommand{\tit}[1]{\textit{#1}}

\newsiamremark{remark}{Remark}
\newsiamremark{hypothesis}{Hypothesis}
\crefname{hypothesis}{Hypothesis}{Hypotheses}
\newsiamthm{claim}{Claim}
\headers{Compression  of Currents and Varifolds}{}

\author{Allen Paul \thanks{University of Bath 
  \email{ap2746@bath.ac.uk}.} \and  Neill Campbell \thanks{University of Bath \email{n.campbell@bath.ac.uk}}
\and Tony Shardlow \thanks{University of Bath 
\email{t.shardlow@bath.ac.uk}.}
}

 \title{Compression of Currents and Varifolds} 

\crefname{section}{Section}{Sections}

\crefname{subsection}{Section}{Sections}

\ifpdf
\hypersetup{
  pdftitle={Compression of Currents and Varifolds},
  pdfauthor={Allen Paul, Neill Campbell, and Tony Shardlow}
}
\fi

\begin{document}

\maketitle

\begin{abstract}
We derive an algorithm for compression of the currents and varifolds representations of shapes, using ridge leverage score (RLS) sampling, and the theory of Nystrom approximation in Reproducing Kernel Hilbert Spaces. Our method is faster than existing compression techniques and comes with theoretical guarantees on the rate of decay of the compression error as a function of the smoothness of the associated shape representation. The obtained compressions are shown to be useful for accelerating downstream tasks such as nonlinear shape registration in the Large Deformation Diffeomorphic Metric Mapping (LDDMM) framework without loss of quality, even for very high compression ratios. The performance of our algorithm is demonstrated on large-scale shape data from modern geometry processing datasets, and is shown to be fast and scalable with rapid error decay.

\end{abstract}



\section{Introduction}
Comparing and computing distances between geometric structures is a fundamental task in computer vision and geometric learning applications. Most commonly, such problems arise when building models of shape variation in a given application domain, for example, in computational anatomy, where shapes are most commonly available as discrete curves and surfaces. When building such models, one requires a metric on shapes in order to best tune the parameters of the model to match observed variation in the data. Ideally, the available shape data for the given task is in parametric correspondence, making the comparison trivial. However, in realistic shape learning tasks, this is far from the case with shape data acquired by completely different methods, with inconsistent parametrisations and differing resolutions across a dataset.


In the case where shape data is available as sub-manifolds of $\R^{d}$, one can deal with the lack of parametric correspondences in a more principled manner, using techniques from geometric measure theory \cite{FidelityMetrics}. In particular, using the so-called currents \cite{GlaunesCurrents}, varifolds \cite{Charon} and normal cycles \cite{NormalCycles} representation of shapes. These representations view shapes as objects that integrate differential forms on the underlying domain. In dimensions $d=2,3$ by choosing these differential forms to lie in a Hilbert space, these representations essentially embed the shapes into a dual space of differential forms. Using the dual metric on these representations allows one to compute the distance between shapes, in terms of their action on forms independent of their parametrisations. The computation of the dual metric can be written down explicitly when the Hilbert space is a Reproducing Kernel Hilbert Space or `RKHS' \cite{RKHS}, induced by a choice of positive-definite kernel function $k:X\times X \longrightarrow \R$. One can also perform the same technique for discrete shape data in a way that is consistent as the resolution of the data tends to infinity. This framework has been used extensively in the LDDMM literature \cite{FidelityMetrics}, for matching shapes with diffeomorphisms, using these shape metrics as the discrepancy term for matching.

However, while the aforementioned shape metrics account for the lack of parametric correspondence when comparing shapes, the practical computation of these metrics can be costly. For example, if one computes the currents metric between two triangulated surfaces with $M$ and $N$ triangles respectively, the metric computation has complexity $\mathcal{O}(MN)$. In modern geometry processing applications, when $M$ and $N$ can exceed $10^{5}$, this becomes far too expensive both memory and computation-wise. A similar issue holds for both the varifolds and normal cycles representations.

In order to deal with the cost of comparing shapes as geometric measures, there are many methods to approximately compute the metric. Approximate methods include fast multipole methods and grid-based FFT methods, which are reviewed in \cite{FidelityMetrics}. However, these methods only approximate the metric computation itself, and not the underlying measures. Furthermore, they often require computations with dense grids on the underlying domain. This works reasonably well for currents with a large spatial scale and shapes without a significant spatial spread. However, for small length-scales and shapes with a large spatial spread, this becomes expensive and results in a large number of grid points. This is exacerbated even more for varifolds, which are measures on a higher-dimensional space than the spatial domain, resulting in far larger grid sizes.

A significant recent development bypassing these issues is the excellent KeOps library \cite{Keops}, which allows large-scale exact metric computation (and exact gradients) through the use of efficient GPU tiling schemes with CUDA and C++. This method allows scalable and \textit{fast} metric computation up to a limit. However, past a certain size (typically $10^{5}$), the KeOps method slows down significantly, which is especially impactful for large-scale shape matching and group-wise registration due to repeated metric and gradient of metric computations. 

\subsection{Contribution}

The main contribution of this work is a fast algorithm for compressing large-scale currents and varifolds associated to shapes, using ridge leverage score (RLS) sampling, and the Nystrom approximation in reproducing kernel Hilbert spaces. The proposed compression algorithm is randomised and much faster than existing compression techniques, showing 10-100 times speed-ups in real-world experiments.  We compress target currents and varifolds of effective size $N,M \geq 10^{5}$, by forming sparse approximations of effective size $n$ and $m$ such that $n\ll N$, $m \ll M$. Post compression, one can compute exact distances between compressed measures at cost $\mathcal{O}(mn)$, which gives significant savings for computing both distance and gradients of the distance. Combined with the speed of compression, this can help massively accelerate subsequent geometric learning algorithms, as we show for LDDMM registration on large-scale surfaces in \cref{Numexpsect}.

Though the algorithm introduces an approximation error, we provide theoretical guarantees on the error that allow us to make estimates of the desired sparsity of compression. Indeed, our main result \cref{OrthogprojlemmaCurr} and \cref{corrcontinuous} together imply an error decay rate of $f(m)=\sum_{i=m+1}^{\infty}{\lambda_{i}} $ at the same rate as the sum of ordered eigenvalues $\lambda_{i}>0$ of the RKHS kernel $k$. In the Gaussian kernel setting, this implies error decay at an exponential rate
\begin{align*}
    \norm{\mu - \hat{\mu}}_{W^{*}}^{2} \leq C\sum_{i=m+1}^{\infty}{\lambda_{i}}  =  \mathcal{O}(m\exp(-\alpha m^{\frac{1}{d}})),
\end{align*}
for some $\alpha > 0$ and where $\hat{\mu}$ is an approximation of ‘size’ m to the original current/varifold $\mu$ in the dual of RKHS $W$. Such estimates hold both in expectation and with high probability with respect to random noisy samples of the original underlying shape.

\subsection{Outline}
In \cref{relatedwork}, we discuss the existing approaches to compression of measure-based shape representations, and compare to our work. Subsequently, we review the currents and varifolds representation of shape in \cref{CurrandVarreview}. We introduce the Nystrom approximation and its RKHS variant in \cref{NystromRKHS}, as well as existing bounds on the error of this approximation. In \cref{mainresults}, we present the proposed compression algorithm, theoretical guarantees and its application to the compression of currents and varifolds, as well as LDDMM matching. In \cref{CompressionofDiscreteMeasures} we present proofs of the main results. Finally, we demonstrate the strengths and weaknesses of the proposed compression algorithm against existing work in \cref{Numexpsect}, on large-scale shape data from modern geometry processing datasets. Finally, we discuss future work and possible extensions in \cref{concplusfut}.

\section{Related work}\label{relatedwork}

\subsection{Compression of currents and varifolds}
The most closely related methods to this work are compression algorithms developed separately for currents and varifolds. To the best of our knowledge, there are three works \cite{CurveCompression,DURRLEMAN,VarCompression} focusing on the compression of such representations of shapes. The first work \cite{CurveCompression} only applies to the case of compression of curve measures, which makes it restricted to specific applications with curve data. Therefore, our work is best compared against the more general works of \cite{DURRLEMAN} and \cite{VarCompression}, which apply to compressing general submanifolds.

The work of \cite{DURRLEMAN} is the first work to consider the compression of large-scale currents. This algorithm is an iterative greedy approach to approximating a target current. At each step, one adds a Dirac delta measure that is most correlated with the current residual measure, which is updated after each step. This algorithm guarantees a decrease in the approximation error and comes with an exponential convergence rate to the target current. However in practice, the computation of the residual and maximisation step tends to be very expensive when running this algorithm on large-scale currents. One can form an approximation to this using grid-based projections and the Fast Fourier Transform. However, the iterative nature of the algorithm and the per-iteration cost mean it can be very slow to compress a given large-scale current with $\sim 10^{5}$ landmarks. As a result, such an algorithm is not suitable when one is interested in computing compressions of multiple currents, or if one wishes to perform the compression routinely as part of an algorithm for shape analysis, in reasonable time scales. Furthermore, the grid-based nature of the algorithm that makes it well suited for currents is a drawback if one wishes to compress varifolds, as the cardinality of the grid representation explodes with the inclusion of the tangential component of varifolds. 

The work of \cite{VarCompression} applies to the compression of varifold measures. This works by fixing a compression level $m$, and minimising the discrepancy between the target varifold and a discrete varifold of size $m$,  using non-convex optimisation techniques applied to the error in the varifolds metric. However, the optimisation itself requires repeatedly computing the varifolds metric to the target measure, which is expensive. Furthermore, the non-convex nature of the associated minimisation problem means there is no guarantee that one finds a good solution without a careful initialisation. This method also does not give any theoretical guarantees on the quality of approximation, which is undesirable. 

In comparison to these existing works, our method, based on the Nystrom approximation, can be used to compress both currents and varifolds in the same framework. Furthermore, our algorithm is not an iterative greedy approach, and is based on randomised projections, meaning we can compress measures in a fraction of the time of the existing methods, while still retaining strong approximation guarantees, via the theoretical bounds afforded to us by the Nystrom approximation. Finally, one is not required to compute the entire target measure in our method, as opposed to the previous existing works. Due to this property, our method scales well to currents and varifolds with $>10^{5}$ landmarks without significant loss of speed and maintaining good approximation guarantees. 
\subsection{Kernel quadrature and interpolation}

We also remark here on the connection to kernel compression methods that have been developed separately, in the context of large-scale kernel learning methods. In particular, the works of \cite{BachKernel, KernquadDPP,KernelInterpCVS, KernelcompressionLyons}. These works all focus on the problem of constructing numerical quadrature rules for integration against probability measures. Such constructions yield discrete probability measures $\hat{\mu} \in W^{*}$ in the dual space of an RKHS $W$ that best approximate a given target probability measure $\mu \in \mathcal{P}(\mathcal{X})$, in the dual norm on $W^{*}$. 

The main technique we use for compression of the measure-based shape representations is an orthogonal projection onto a subspace generated by a randomised data-dependent distribution. In our case, the data-dependent distribution is defined by approximated ridge leverage scores. We bound the compression error of doing so, in terms of the interpolation error between two associated dual functions in the RKHS. 

The idea of interpolation/quadrature approximation of RKHS functions via orthogonal projection onto randomised sub-spaces has been explored in \cite{KernquadDPP} and \cite{KernelInterpCVS}, where the quadrature nodes (or control points as in this paper) are chosen via a repulsive point distribution. In \cite{KernquadDPP}, this distribution is a projection Determinantal Point Process (DPP), and in \cite{KernelInterpCVS}, this distribution is a continuous volume-sampling distribution, which is shown to be a mixture of projection DPPs. However, these approximation schemes require access to an eigen-decomposition of the kernel of the RKHS, or in the discrete case, an eigen-decomposition of the full kernel matrix. One proposed way to overcome this is via MCMC methods for DPP in \cite{KernelInterpCVS}, which eliminates the need for eigen-decomposition, but is still very slow for large-scale problems. 

While we obtain similar interpolation bounds in the specific case of discrete currents and varifolds, we do so in a way that does not depend on using the DPP distribution. Instead, we derive a bound in terms of trace error of the Nystrom approximation by leveraging connections to sparse interpolation results of \cite{connections}. From here, we are free to apply any theoretically sound sampling technique for the Nystrom approximation, such as RLS sampling, DPP, Greedy initialisation, etc. This step allows us to use the latest developments in RLS sampling, such as \cite{DAC,Musco}, in order to form fast, low-dimensional approximations that are orders of magnitude faster than DPP sampling, while still retaining comparable theoretical bounds.

The recent work of \cite{KernelcompressionLyons} on kernel quadrature does not have the same restrictions as \cite{KernquadDPP},\cite{KernelInterpCVS}, and has competitive run-times for compression of probability measures, with similar error bounds. Furthermore, this method shows promise for large-scale compression, as evidenced by numerical examples on empirical measures from machine learning datasets. However, this work applies to the case of probability measures and not to general linear functionals on an RKHS. Therefore, in the case of currents and normal cycles, which yield non-positive linear functionals on vector-valued functions, this method is not directly applicable. However, it may be interesting to compare the performance of our method against that of \cite{KernelcompressionLyons} for varifold compression, as a varifold can straightforwardly be written as a probability measure, provided the kernel is sufficiently fast-decaying, e.g. a Gaussian kernel.

\section{Currents and varifolds}\label{CurrandVarreview}

We now briefly review the currents and varifolds representation of shape, and their applications in computational geometry. In particular, their application to computing a metric distance between shapes. In the rest of this paper, we present the case for smooth surfaces $S \subset \R^{d}$, with $d=3$, but the analogous statements hold for smooth curves $C$ with $d=2,3$. Furthermore, we focus on the treatment of currents and varifolds in terms of their action on vector fields and functions. In full generality, this may be extended to action on differential forms and integration with respect to Hausdorff measure, but we do not pursue this route here. See \cite{DURRLEMAN,Charon,FidelityMetrics} for a treatment of this perspective.
\subsection{Currents}\label{Currentssec}

Given a smooth surface $S \subset \R^{d}$, one can define a continuous linear functional $\mu_{S}$ on continuous bounded vector fields $v \in C_{0}(\R^{d},\R^{d})$ as
\begin{align}\label{CurrentAction}
    \mu_{S}(v) := \int_{S}\langle v(x),n(x)\rangle dS(x),
\end{align}
where $n(x)$ denotes the outward unit normal at $x\in S$, and $dS(x)$ denotes surface measure. The action \eqref{CurrentAction} is simply the surface integral of $v$ along $S$. This functional $\mu_{S}$ is the \tit{current} \cite{GlaunesCurrents} associated to the curve $S$. Introducing a test Reproducing Kernel Hilbert Space (RKHS) \cite{RKHS} of continuous bounded vector fields $V$ with matrix-valued kernel $K:\R^{d}\times\R^{d}\longrightarrow \R^{d\times d}$, one embeds $\mu_{S}$ into the dual space $V^{*}$ due to the embedding $V \hookrightarrow C_{0}(\R^{d},\R^{d})$. A typical choice is a scalar diagonal kernel $K(x,y)=k(x,y)I_{d}$ for $k:\R^{d}\times\R^{d}\longrightarrow \R$; for example, with $k$ a Gaussian radial basis function (RBF) kernel or Cauchy kernel, both of which are standard choices \cite{FidelityMetrics}.  A natural choice of norm on the dual space $V^{*}$ is the dual norm
\begin{align*}
    \norm{\mu_{S}}_{V^{*}} = \underset{\underset{\norm{v}\leq 1}{v\in V}}{\sup}{\lvert \mu(v)\rvert}.
\end{align*}
Using the reproducing property of the kernel $k$ on $V$, one can prove \cite{GlaunesCurrents} that 
\begin{align*}
    \norm{\mu_{S}}_{V^{*}}^{2} =  \int_{S}\int_{S}{{k(x,y )\langle n(x),n(y)\rangle }dS(x)dS(y)}.
\end{align*}
The dual norm induces the dual metric, which is simply the metric induced by the norm. Given two elements $\mu,\kappa \in V^{*}$, one can compute the dual metric as
\begin{align*}
    d_{V^{*}}(\mu,\kappa) = \norm{\mu-\kappa}_{V^{*}} = \underset{\underset{\norm{v}\leq 1}{v\in V}}{\sup}{\lvert \mu(v) - \kappa(v)\rvert}.
\end{align*}
Given two smooth surfaces $S_{1},S_{2} \subset \R^{d}$, one can therefore compute distances between the shapes in the dual metric
\begin{align*}
   d(S_{1},S_{2}) := d_{V^{*}}(\mu_{S_{1}},\mu_{S_{2}}) = \norm{\mu_{S_{1}} -  \mu_{S_{2}}}_{V^{*}} = \underset{\underset{\norm{v}\leq 1}{v\in V}}{\sup}{\lvert \mu_{S_{1}}(v) - \mu_{S_{2}}(v)\rvert}.
\end{align*}
Intuitively, this distance measures how differently the shapes integrate the same vector fields $v\in V$ along their surface, ranging over the whole space $V$. Therefore, the global geometry of the shapes is distinguished by test vector fields that `probe' the shapes via the surface integral. Of course, the regularity and frequency of the vector fields in $V$ determine the characteristic scale at which shapes are distinguished in such metrics. For surfaces, the dual metric has the form
\begin{gather*}
\norm{\mu_{S_{1}} -  \mu_{S_{2}}}_{V^{*}}^{2} =\norm{\mu_{S_{1}}}_{V^{*}}^{2} - 2\langle \mu_{S_{1}},\mu_{S_{2}}\rangle_{V^{*}} + \norm{\mu_{S_{2}}}_{V^{*}}^{2},\\
\end{gather*}
which can be computed using the explicit formulae
\begin{gather*}
    \langle \mu_{S_{1}},\mu_{S_{2}}\rangle_{V^{*}} := \int_{S_{1}}\int_{S_{2}}{k(x,y )\langle n_{1}(x),n_{2}(y)\rangle dS_{1}(x)dS_{2}(y)},\quad \norm{\mu_{S_{1}}}_{V^{*}}^{2} = \langle \mu_{S_{1}},\mu_{S_{1}}\rangle_{V^{*}},
\end{gather*}
where the former term can be interpreted as an inner product, see \cite{FidelityMetrics}. Notice that computation of this metric between $S_{1}$ and $S_{2}$ does not require \tit{any} correspondence information, only access to the surfaces and normals. This makes the currents representation ideal for metric comparison of shape, as it does not require any further learning to extract dense parametric correspondences, and gives explicit formulae for the metric that gives a measure of the difference in global geometry. 

Of course in practice, one only has access to a discretisation of the smooth surfaces of interest $S_{1}, S_{2} \subset \R^{d}$. Most commonly, surfaces may be available as a triangulation $\hat{S}_{1}=\{T_{i}^{1}\}_{i=1}^{N}$, $\hat{S}_{2}=\{T_{i}^{2}\}_{i=1}^{M}$. In order to make geometric comparisons between the discretised shapes, we can once again embed the observed data into the dual space $V^{*}$ to obtain a discrete sum of Dirac functionals
\begin{align*}
    \hat{S}_{1} \hookrightarrow \hat{\mu}_{S_{1}}= \sum_{i=1}^{N}{\delta_{c_{i}^{1}}\nu_{i}^{1} } \in V^{*} ,\quad \hat{S}_{2} \hookrightarrow \hat{\mu}_{S_{2}} = \sum_{j=1}^{N}{\delta_{c_{j}^{2}}\nu_{j}^{2} } \in V^{*},
\end{align*}
where $c_{i}^{1}$, $\nu_{i}^{1} \in \R^{d}$  denotes the centre and \tit{unnormalised} normal vector, respectively, of the $i$'th triangle $T_{i}$ of $\hat{S_{1}}$. Both the centres and normals can be computed from the vertices and edges of the discrete triangles using averages and cross products. Each weighted Dirac functional $\delta_{a}b \in V^{*}$ in the above, has the action $(\delta_{a}b|v) = b^{T}v(a)$ for all $v \in V$. This gives a simple action for the discrete functionals. For example for $\hat{S}_{1}$,
\begin{align*}
    \hat{\mu}_{S_{1}}(v)= \sum_{i=1}^{N}{(\nu_{i}^{1})^{T}v(c_{i}^{1}) }.
\end{align*}
It is possible to show under suitable regularity constraints \cite{ConvergenceofVarCurr1,ConvergenceofVarCurr2,FidelityMetrics} the discrete representation is consistent with the continuous case, in the sense that $\norm{ \hat{\mu}_{S_{1}}-\mu_{S_{1}}}_{V^{*}} \leq C\tau(N)$ such that $\tau(N)\to 0$ as $N \to \infty$. Therefore, for a sufficiently fine discretisation, we do not lose much geometric information by working with $\hat{\mu}_{S_{1}}$.

Once again, we can compute the dual metric between $\hat{\mu}_{S_{1}}, \hat{\mu}_{S_{2}}$ in order to make geometric comparisons without correspondences. Furthermore in the discrete case, using the reproducing property of the RKHS, one can show \cite{GlaunesCurrents} the dual metric has a very simple expression that can be computed \tit{exactly}, using kernel evaluations and normal vector inner products:
\begin{gather}
  \begin{split}
    d_{V^{*}}( \hat{\mu}_{S_{1}}, \hat{\mu}_{S_{2}})^{2} &= \sum_{i,j=1}^{N}{k(c_{i}^{1},c_{j}^{1})\langle \nu_{i}^{1},\nu_{j}^{1}\rangle} - 2\sum_{i=1}^{N}\sum_{j=1}^{M}{k(c_{i}^{1},c_{j}^{2})\langle \nu_{i}^{1},\nu_{j}^{2}\rangle}\\&\qquad + \sum_{i,j=1}^{M}{k(c_{i}^{2},c_{j}^{2})\langle \nu_{i}^{2},\nu_{j}^{2}\rangle}.
  \end{split}
  \label{DiscreteCurrMet}
\end{gather}
This is easy to implement and to compute gradients of, for downstream tasks such as shape registration using the currents metric as a discrepancy term. 

\begin{remark}
     Note that the presentation of this section can also easily be extended to continuous and discrete curve data in $\R^{d}$ for $d\in\{2,3\}$ as well. This is done by replacing the surface integrals with line integrals, and formulating the analogous discrete formulation with discrete tangent vectors and segment centres. See \cite{CurveLDDMM} for the specifics.
\end{remark}

However, in modern geometry processing tasks where discretised shapes $\hat{S_{1}},\hat{S_{2}}$ are available with resolutions $N,M \geq 10^{4} $, one has to be careful in the method used for computing the discrete currents metric presented in \cref{DiscreteCurrMet}. The computational and memory cost of the discrete metrics scales as $\mathcal{O}(NM)$, which makes the naive implementation infeasible for the large $N,M$ regime, without specialised hardware, or tiling schemes such as implemented in Keops \cite{Keops}. 

\subsection{Varifolds}\label{Varifolds}

In the varifolds representation of shape \cite{Charon}, shapes are treated as linear functionals on a space of {real-valued} functions lying in an RKHS $V$ consisting of functions of the form:
\begin{align*}
    f:\mathcal{X} \longrightarrow \R ,\quad \mathcal{X}=\R^{d}\times \mathbb{S}^{d-1},
\end{align*}
induced by a real-valued kernel function $K:\mathcal{X} \times \mathcal{X} \longrightarrow \R$. Typically, this is a product kernel over the two components of $\mathcal{X}$, of the form
\begin{align*}
    K((x,\nu),(y,\mu)) = K_{p}(x,y)K_{s}(\nu,\mu),\quad (x,\nu),(y,\mu) \in \mathcal{X},
\end{align*}
that factorises over spatial and spherical components, as $K_{p}$ and $K_{s}$, respectively. 

As in \cref{Currentssec}, smooth surfaces $S\subset \R^{d}$ are embedded into the dual space $V^{*}$ via surface integral action
\begin{align*}
    \mu_{S}(v) = \int_{S}{v(x,n(x))dS(x)}.
\end{align*}
The functional $\mu_{S}$ is known as the varifold associated to $S$. Once again, given two smooth surfaces $S_{1},S_{2}\subset \R^{d}$ to be compared, one can compute geometric distances between them in the dual metric
\begin{align*}
    d(S_{1},S_{2}) := d_{V^{*}}(\mu_{S_{1}},\mu_{S_{2}}) = \norm{\mu_{S_{1}} - \mu_{S_{2}}}_{V^{*}}.
\end{align*}
A similar intuition as \cref{Currentssec} holds, that metric comparison of shapes is made by comparing how differently shapes integrate the same functions in the RKHS $V$  along their surface. However, one must note that the test functions now vary both in space and normal components. As we shall see in the metric expression, this means the varifold gives a strictly richer representation of shape.

As in \cref{Currentssec}, it is possible to show closed-form expressions for the dual metric when applied to varifolds associated to shapes. In particular,
\begin{align*}
  \norm{\mu_{S_{1}} -  \mu_{S_{2}}}_{V^{*}}^{2} = \norm{\mu_{S_{1}}}_{V^{*}}^{2} - 2\langle \mu_{S_{1}},\mu_{S_{2}}\rangle_{V^{*}} + \norm{\mu_{S_{2}}}_{V^{*}}^{2},
\end{align*}
where the individual terms are explicitly computed as
\begin{align*}
    \langle \mu_{S_{1}},\mu_{S_{2}}\rangle_{V^{*}} := \int_{S_{1}}\int_{S_{2}}{  K_{p}(x,y)K_{s}(n_{1}(x),n_{2}(y))dS_{1}(x)dS_{2}(y)},\quad \norm{\mu_{S_{1}}}_{V^{*}}^{2} = \langle \mu_{S_{1}},\mu_{S_{1}}\rangle.
\end{align*}
In the case where the spherical kernel is linear, so that $K_{s}(u,v)=\langle u,v\rangle$, then the above metric reduces to the currents metric between shapes. However, in general one can choose \tit{any} nonlinear valid kernel on the sphere, such as the Gaussian kernel. In this case, the expression of the inner-product term makes clear why the varifolds representation is often preferred over currents. In particular, the varifolds metric performs nonlinear kernel comparison of normal vectors along the surface, whereas the currents metric is restricted to linear comparison. The extra nonlinearity in the normal component means there are no cancellation effects for varifolds as opposed to currents, and allows a richer metric comparison of shapes.

One can perform a similar comparison for discrete shapes, using the same approach as in \cref{Currentssec}. Given triangulated shapes $\hat{S}_{1}=\{T_{i}^{1}\}_{i=1}^{N}$, $\hat{S}_{2}=\{T_{i}^{2}\}_{i=1}^{M}$, one defines the associated discrete varifolds in $V^{*}$ as
\begin{align*}
   \mu_{\hat{S}_{1} } = \sum_{i=1}^{N}{\delta_{(c_{i}^{1},n_{i}^{1})}\norm{\nu_{i}^{1}}},\quad \mu_{\hat{S}_{2} } = \sum_{i=1}^{M}{\delta_{(c_{i}^{2},n_{i}^{2})}\norm{\nu_{i}^{2}}},
\end{align*}
where the notation for $c_{i}^{k},\nu_{i}^{k}$ are the same as \cref{Currentssec} and $n_{i}^{k}$ denotes \tit{unit} normal vector of triangle $T_{i}$. As for currents, under suitable regularity constraints on the mesh refinement \cite{ConvergenceofVarCurr1,ConvergenceofVarCurr2,FidelityMetrics,Charon} the discrete representation is consistent with the continuous representation, in the sense that $\norm{ \hat{\mu}_{S_{1}}-\mu_{S_{1}}}_{V^{*}} \leq C\tau(N)$ such that $\tau(N)\to 0$ as $N \to \infty$, for smooth $S$.

Finally, the reproducing property of the RKHS kernel yields simple closed-form formulae for the varifolds metric between two discrete shapes
\begin{align*}
 d_{V^{*}}( \hat{\mu}_{S_{1}}, \hat{\mu}_{S_{2}})^{2} &= \sum_{i,j=1}^{N}{K_{p}(c_{i}^{1},c_{j}^{1})K_{s}(n_{i}^{1},n_{j}^{1})\norm{\nu_{i}^{1}}\norm{\nu_{j}^{1}}} \\\quad&- 2\sum_{i=1}^{N}\sum_{j=1}^{M}{K_{p}(c_{i}^{1},c_{j}^{2})K_{s}(n_{i}^{1},n_{j}^{2})\norm{\nu_{i}^{1}}\norm{\nu_{j}^{2}}  }\\ &\quad+ \sum_{i,j=1}^{M}{K_{p}(c_{i}^{2},c_{j}^{2})K_{s}(n_{i}^{2},n_{j}^{2})\norm{\nu_{i}^{2}}\norm{\nu_{j}^{2}} }.
\end{align*}

The analogous statements of this section for varifolds associated with curves in $\R^{d}$ for $d\in \{2,3\}$, also hold, see \cite{FidelityMetrics}. Notice once more that in the large $N,M$ regime, naive computation of the varifold metric becomes infeasible both computationally and memory-wise, and requires specialised routines and GPU implementations such as in the Keops library.

\section{Nystrom approximation in RKHS}\label{NystromRKHS}
We now discuss the main tool we use to derive a compression algorithm: the Nystrom approximation in RKHS.

Suppose we have a regression dataset $\{(x_{1},y_{1}),\dots,(x_{n},y_{n})\} \subset \mathcal{X} \times \R$ for a domain $\mathcal{X}$, on which we perform Kernel Ridge Regression (KRR) with positive-definite kernel $k$. Kernel ridge regression is the following regularised minimisation scheme over the RKHS $H_{k}$ of real-valued functions induced by the kernel $k$:
\begin{align}\label{KRRProblem}
    \hat{f} = \underset{f \in H_{k}}{\mathrm{argmin}}{\sum_{i=1}^{n}(y_{i}-f(x_{i}))^{2} + \lambda\norm{f}_{H_{k}}^{2} }.
\end{align}
There exists a unique minimiser to this problem \cite{connections}, and it takes the form
\begin{align*}
    \hat{f} = \sum_{i=1}^{n}{\alpha_{i}k(\cdot,x_{i})},
\end{align*}
where the coefficient vector ${\alpha} = (\alpha_{1},\dots,\alpha_{n})^{T} \in \R^{n}$ is given by
\begin{align*}
    {\alpha} = (K_{XX} + \lambda I_{n})^{-1}{y},\quad {y}=(y_{1},\dots,y_{n})^{T}\in \R^{n},
\end{align*}
where $(K_{XX}) \in \R^{n\times n}$ is the kernel matrix evaluated on input points $\{x_{i}\}_{i=1}^{n}$.

The \tit{Nystrom} KRR \cite{connections} is a related problem that fits the data in a restricted subspace of the RKHS, that is parametrised by a finite set of $m$ distinct control points $\bm{c}=\{c_{i}\}_{i=1}^{m}\subset \mathcal{X}$ in the domain. Defining the subspace
\begin{align*}
    M := \bigg\{\sum_{j=1}^{m}{\alpha_{j}k(\cdot,c_{j}) }: \alpha_{i} \in \R\bigg\} = \mathrm{span}\{k(\cdot,c_{1}),\dots,k(\cdot,c_{m})\},
\end{align*}
the Nystrom KRR problem is to find
\begin{align}\label{Nystromsol}
    \bar{f} = \underset{f \in M}{\mathrm{argmin}}{\sum_{i=1}^{n}(y_{i}-f(x_{i}))^{2} + \lambda\norm{f}_{H_{k}}^{2} }.
\end{align}
Once again, there exists a unique solution to this problem of the form
\begin{align*}
    f = \sum_{i=1}^{m}{\beta_{i}k(\cdot,c_{i})},
\end{align*}
where the coefficients ${\beta} = (\beta_{1},\dots,\beta_{m})^{T} \in \R^{m}$ are given by
\begin{align*}
    {\beta} = (K_{CX}K_{XC} + \lambda K_{CC})^{-1}K_{CX}{y}.
\end{align*}
where $K_{CC} \in \R^{m\times m}$ is defined analogously to the full kernel matrix, and $(K_{XC})_{ij}=k(x_{i},c_{j})$. The following theorem of \cite{connections} provides a bound on the error between the full solution and the $m$-dimensional Nystrom approximation to the KRR problem.

\begin{theorem}[\cite{connections}]\label{thmox}
Let $k: \mathcal{X} \times \mathcal{X} \longrightarrow \R_{\geq 0}$ be a symmetric positive-definite kernel function with associated RKHS $H_{k}$. Let $\hat{f}$, $\bar{f}$ be the KRR and Nystrom KRR solutions, respectively, for data $\{(x_{i},y_{i})\}_{i=1}^{n}\subset\mathcal{X} \times \R$, $m$ distinct control points $\bm{c}=\{c_{i}\}_{i=1}^{m}$, and $\lambda>0$ fixed. Then, the following bound holds
\begin{align*}
    \norm{\hat{f}-\bar{f}}_{H_{k}}^{2} \leq \frac{2\mathrm{tr}(K_{XX}-Q_{XX})\norm{y}^{2}}{\lambda^{2}},\quad Q_{XX}:=K_{XC}K_{CC}^{-1}K_{CX}.
\end{align*}
\end{theorem}
The matrix $Q_{XX}$ in the previous theorem is the well-known Nystrom approximation \cite{OGNystrom, SpectralMethods} to the kernel matrix $K_{XX}$.

\subsubsection{Nystrom approximation for kernel matrices}\label{NystromMatrices}

The Nystrom approximation method can be applied to general positive-semidefinite matrices, but is most frequently used in forming a low-rank approximation to kernel matrices. Suppose that we are given $n$ points $X=\{x_{i}\}_{i=1}^{n}\subset\mathcal{X}$ in the domain, and $k:\mathcal{X} \times \mathcal{X} \longrightarrow \R$ is a positive semi-definite kernel function. Defining $(K_{XX})_{ij} = k(x_{i},x_{j})$, how can we best approximate the kernel matrix $K_{XX} \in \R^{n\times n}$ with a low-rank matrix? It is well known from finite-dimensional linear algebra that the best rank-$m$ approximation to the kernel matrix is given by the SVD truncation, where optimality is measured by error in Frobenius norm. However, the computation of such an approximation requires computing an eigen-decomposition, which scales as $\mathcal{O}(n^{3})$. For very large $n$, as frequently occurs in machine learning, this is an obstacle. 


In the Nystrom approximation \cite{OGNystrom} for kernel matrices, one subsamples $m$ distinct control points $\bm{c}=\{c_{i}\}_{i=1}^{m}\subset X=\{x_{i}\}_{i=1}^{n}$, and forms a \tit{Nystrom approximation} to the full matrix as,
\begin{align*}
    K_{XX} \approx K_{XC}K_{CC}^{-1}K_{CX} =: Q_{XX}.
\end{align*}
 It is possible to obtain an error bound for this approximation \cite{SpectralMethods}, for a suitable sampling distribution of the control points. 
\begin{theorem}[\cite{SpectralMethods}]\label{eigenvaluethm}
Suppose that $m$ distinct control points $\bm{c}=\{c_{i}\}_{i=1}^{m}$, are drawn from $X$ with probability $p(\bm{c}) \propto det(K_{CC})$ that is proportional to the determinant of the cross evaluation matrix. Then, the following bound holds in expectation with respect to $p(\bm{c})$:
\begin{align*}
  \E_{\bm{c}}[\mathrm{tr}(K_{XX} - Q_{XX})] \leq (m+1)\sum_{i=m+1}^{n}{\lambda_{i}(K_{XX})},
\end{align*}
where $\lambda_{i}(K_{XX})$ denotes the $i$'th largest eigenvalue of the full kernel matrix.
\end{theorem}
See \cite{SpectralMethods} for a proof. The distribution of control points such that $p(\bm{c}) \propto \det(K_{CC})$ is known as a \textit{discrete determinantal point process (DPP)}. For a fixed number of control points $m$, it is also referred to as a discrete $m$-DPP \cite{SpectralMethods}.

\cref{eigenvaluethm} demonstrates that with respect to draws from the $m$-DPP, the rate of decay of the Nystrom matrix approximation error is at least as fast as the eigenvalue decay of the target matrix. The eigenvalue decay of the kernel matrix is, of course, closely connected to the properties of the kernel function. For example, \cite{ratesofconvergesparsegp} proves the following result.
\begin{theorem}[\cite{ratesofconvergesparsegp}]
Let $p$ be a probability density function over the domain $\mathcal{X}$. Denote $X=\{x_{i}\}_{i=1}^{n}$ a size $n$ i.i.d sample, from this distribution. Taking expectation over with respect to draws from this distribution gives the following bounds:
\begin{align}\label{eigdecay}
  \E_{X}\bigg[\sum_{i=m+1}^{n}{\lambda_{i}(K_{XX})}\bigg] \leq n
  \sum_{i=m+1}^{\infty}{\lambda_{m}}
\end{align}

where the eigenvalues $\lambda_{i}$ in the right-hand side sum are the decreasingly ordered eigenvalues of the kernel integral operator $\mathcal{K}_{p}: L^{2}(\mathcal{X},pdx)\longrightarrow L^{2}(\mathcal{X},pdx)$ defined as follows,
\begin{align*}
    (\mathcal{K}_{p}g)(x') = \int_{\mathcal{X}}{g(x)k(x,x')p(x)dx},
\end{align*}
where one defines
\begin{align*}
L^{2}(\mathcal{X},p dx) := \bigg\{f:\mathcal{X}\longrightarrow \R \ :\ \int_{\mathcal{X}}{f(x)^{2}p(x)dx} < +\infty\bigg\} ,
\end{align*}
up to $p$ almost-everywhere equivalence.
\end{theorem}
The previous bounds indicate that with respect to draws from the $m$-DPP and the input distribution, the average error of the matrix Nystrom approximation decays at the rate of the eigenvalue decay of the kernel integral operator. For common kernels such as the Gaussian RBF, one can analytically bound the eigenvalue decay \cite{ratesofconvergesparsegp} as follows 
\begin{align}\label{theoreticaleigdecay}
    \sum_{i=m+1}^{\infty}{\lambda_{i}} = \mathcal{O}(m\exp(-\alpha m^{\frac{1}{d}})),
\end{align}
for some $\alpha>0$ that is an increasing function of the kernel length-scale. 

Therefore, the larger the length-scale, the faster the eigenvalue decay and the better the Nystrom approximation for small $m$. For kernels with smaller length-scales and therefore slower eigenvalue decay, a higher $m$ will be required to make the average error smaller. 

While sampling from the $m$-DPP gives a convergence bound, sampling naively from this distribution is computationally intractable for large $n$ and moderate $m$. While there exist more sophisticated algorithms for sampling exactly from such distributions, they require computing an eigen-decomposition of the kernel matrix with cost $\mathcal{O}(n^{3})$, which is also computationally infeasible. One way to overcome this issue while retaining the theoretical guarantees is via MCMC sampling to sample from an $m$-DPP. The sampling scheme is given in \cref{MCMKDPP}.

\subsection{Ridge leverage score sampling for the Nystrom approximation}
An alternative, much faster method to ensure good quality of Nystrom approximation is the ridge leverage score (RLS) sampling technique for control points. In the notation of \cite{Musco}, given a kernel matrix $K:=K_{XX}$ and a fixed level of approximation $\lambda$, the RLS sampling technique, samples each point $x_{i} \in X$ with probability 
\begin{align*}
    l_{i}(\lambda)(K) = (K(K+\lambda I)^{-1})_{i,i},\quad i=1,\dots,n
\end{align*}
known as the ridge leverage score of the point $x_{i}$, and constructs the Nystrom approximation on the subsampled points $\bm{c}=\{c_{i}\}_{i=1}^{m}\subset X$. Denoting $\tilde{K}$ as the Nystrom approximation on these control points, it can be shown \cite{Musco} that the RLS sampling procedure gives the error bound
\begin{align*}
    \norm{K- \tilde{K} }_{2} \leq \lambda 
\end{align*}
with probability $1-\delta$, for randomised sample size satisfying $m=\mathcal{O}(d^{\lambda}\log(d^{\lambda}/\delta))$, where $d^{\lambda}:=\mathrm{tr}(K(K+\lambda I)^{-1})$. Note that we also have that the size $m$ is randomised as a function of $\lambda$, the error. If one sets $\lambda = \frac{1}{k}\sum_{k+1}^{n}{\lambda_{i}(K)}$, it may be shown \cite{Musco} that $d^{\lambda} = \mathcal{O}(k)$, so that 
\begin{align}\label{fulleigdecay}
    \norm{K- \tilde{K}}_{2} \leq \frac{1}{k}\sum_{k+1}^{n}{\lambda_{i}(K)}
\end{align}
with probability $1-\delta$ and randomised sample size $m=\mathcal{O}(k\log(k/\delta))$. Alternatively, in practice, one may sample exactly $m = \lceil (k\log(k/\delta))\rceil$ samples directly from the leverage score distribution without replacement, so that the above bounds hold approximately. This is often preferable in practice as the constants involved in the upper bound are large and may result in oversampling relative to a fixed approximation level.

Of course, computing the RLS scores $\{l_{i}\}_{i=1}^{n}$ is computationally infeasible due to the prohibitively large kernel matrices and inverses involved. As a result, there are a wide array of algorithms that have been developed in order to approximately sample from this distribution while still maintaining the error bounds up to a factor. Such algorithms require significantly less memory and computational cost, making them the algorithm of choice in practical situations for the Nystrom approximation. In this work, we consider two such sampling methods with complexity $\mathcal{O}(m^{2}n)$; the recursive method of \cite{Musco} and, more recently, the method of \cite{DAC}. 

For example, Algorithm $3$ of \cite{Musco} gives the following theoretical guarantee.
\begin{theorem}[\cite{Musco}]
    Fix $\delta \in (0,\frac{1}{32})$, and $S \in \N$. There exists a constant $c>0$ such that the recursive RLS sampling algorithm of \cite{Musco}, with probability $1-3\delta$ returns $m\leq cS\log(S/\delta)$ distinct control points $\bm{c}=\{c_{i}\}_{i=1}^{m}\subset X$ such that
    \begin{align}\label{RLSIneq}
        \norm{K - \tilde{K}}_{2} \leq \frac{1}{S}\sum_{i=S+1}^{n}{\lambda_{i}(K)}.
    \end{align}
\end{theorem}
If one wants to construct an approximation with a number of control points approximately $m$, it suffices to choose $S$ such that $m\approx S\log(S/\delta)$, which would give the error bound on the right-hand side with high probability. Sampling using this algorithm is significantly faster than using MCMC samplers, and in practice gives good quality samples that ensure the Nystrom approximation remains small with strong theoretical guarantees. In practice, one may sample exactly $m =\lceil (S\log(S/\delta))\rceil$ samples directly from the approximate leverage score distribution without replacement, so that the above bounds hold approximately. This is implemented in \footnote{https://github.com/cnmusco/recursive-nystrom}, made available by the authors of \cite{Musco}. Again, we note that the constants in the upper bound of $m$ are large, and in practice, it suffices to take $m\approx (S\log(S/\delta))$.

Another efficient RLS approximation method is the divide-and-conquer (DAC) RLS approximation, which is Algorithm $2$ of \cite{DAC}. This algorithm also yields the same bound \cref{fulleigdecay} but with probability $1-\delta$ and $m \leq c(S + n(1-\alpha))\log(\frac{S + n(1-\alpha)}{\delta} )$, where $\alpha \in (0,1)$. The theory bound on $m$ here is looser than \cite{Musco} with the additive factor of $n(1-\alpha)$. While the theoretical guarantees of \cite{DAC} algorithm are not quite as tight as that of \cite{Musco}, in practice one observes that sampling $m\approx S\log(S/\delta)$ from the DAC leverage scores yields Nystrom approximations of similiar error to the method of \cite{Musco}, while being significantly faster in practice; especially in the case of $n\geq 10^{5}$. Furthermore, it is very easy to parallelise the DAC algorithm, in contrast to the recursive RLS scheme, thus gaining a further speedup. Therefore, even though we prove bounds in section \cref{mainresults} for the method of \cite{Musco}, the DAC method of \cite{DAC} is usually our algorithm of choice in practice for sampling control points for compression of large-scale currents and varifolds.

\section{Main results}\label{mainresults}

We now propose a method to compress discrete linear functionals of the form
\begin{align*}
    \mu = \sum_{i=1}^{n}{\delta_{x_{i}}\alpha_{i}} \in V^{*},\quad (x_{i},\alpha_{i}) \in \mathcal{X} \times \R^{p},
\end{align*}
where we take $\mathcal{X}$ to be a non-empty set, and assume a fixed RKHS $V \subset C(\mathcal{X},\R^{p} )$ induced by a scalar diagonal kernel $K(x,y) = k(x,y)I_{p}$. This choice of scalar diagonal kernel is by far the most popular in practical applications of currents and varifolds. In this section, we derive a method for a general $p\in \N$, and we shall apply this result to the special case of currents and varifolds in subsequent sections. The main techniques we employ are the Nystrom approximation and interpolation theory in RKHS, which allows us to obtain fast-decaying error bounds for the compression, while also giving a computationally cheaper approximation method.  

The compressed approximation to the full functional will take the form
\begin{align}\label{Deltaapprox}
   \hat{\mu} =  \sum_{i=1}^{m}{\delta_{c_{i}}\beta_{i} } ,\quad (c_{i},\beta_{i})\in \mathcal{X}\times\R^{p},
\end{align}
for appropriately chosen $c_{i},\beta_{i}$ and $m\ll n$ such that the error is bounded above by a small constant:
\begin{align}\label{errorbound}
    \norm{\mu-\hat{\mu}  }_{V^{*}} < \epsilon.
\end{align}
Note that we wish to approximate in the discrete delta form \eqref{Deltaapprox}, as it will allow us to use the exact dual metric computation formula \cref{DiscreteCurrMet} on the compressed form, with $m\ll n$ terms. This reduces dual metric computation cost, from $\mathcal{O}(n^{2})$ to $\mathcal{O}(m^{2})$, which gives significant saving when $m\ll n$.

Henceforth, we denote by $L: V \longrightarrow V^{*}$ the isometric duality mapping of the RKHS $V$. By the reproducing property of the kernel functions, it is a standard fact that the following equality holds:
\begin{align}\label{dualityidentity}
    L^{-1}({\delta_{c}}w ) =   k(\cdot,c )w \in V,\quad (c,w) \in \mathcal{X}\times \R^{p},
\end{align}
which is the dual vector field in $V$ to the linear functional $\delta_{c}w$. Using the Riesz representation theorem and the explicit form \cref{dualityidentity} of the isometric duality map in RKHS, the proposed approximation of discrete functionals is equivalent to approximating functions in $V$ of the form:
\begin{align*}
    v^{\alpha}(x) := \sum_{i=1}^{n}{k(x,x_{i})\alpha_{i}} =  L^{-1}\bigg(\sum_{i=1}^{n}{\delta_{x_{i}}\alpha_{i}} \bigg),
\end{align*}
with compact approximations of the form,
\begin{align*}
     v_{m}(x) := \sum_{i=1}^{m}{k(x,c_{i})\beta_{i}} = L^{-1}\bigg(\sum_{i=1}^{m}{\delta_{c_{i}}\beta_{i}} \bigg),
\end{align*}
with size $m\ll n$.
\subsection{Algorithms and theoretical guarantees}
We now describe our proposed algorithm for compressing discrete functionals in RKHS in the following. 
\begin{algorithm}[H]
\caption{Discrete functional compression}\label[algorithm]{FunctionalCompressionAlg}
\begin{algorithmic}[1]
\State Fix input domain $\mathcal{X}$, $m \ll n$, RKHS $V$, scalar kernel function $k:\mathcal{X}\times \mathcal{X} \longrightarrow \R$ , a sampling algorithm for control points, and target functional
\begin{align*}
    \mu = \sum_{i=1}^{n}{\delta_{x_{i}}\alpha_{i}},\quad (x_{i},\alpha_{i}) \in \mathcal{X}\times \R^{p}.
\end{align*}

\State Subsample $m$ distinct control points $\{c_{i}\}_{i=1}^{m} \subset \{x_{i}\}_{i=1}^{n}$ from the chosen sampling algorithm for control points.
\State Evaluate the dual vector valued function $L^{-1}(\mu)$ on the control points as
\begin{align*}
y_{j} = \sum_{i=1}^{n}{k(c_{j},x_{i})\alpha_{i}},\quad j=1,\dots,m.
\end{align*}
\State Form an approximation $\hat{\mu}$ to $\mu$ via orthogonal projection onto the control-point subspace
\begin{align*}
    \hat{\mu} = \sum_{i=1}^{m}{\delta_{c_{i}}\beta_{i}},\quad \beta = [\beta_{1},\dots,\beta_{m}]^{T}= K_{CC}^{-1}y \in \R^{m\times p}.
\end{align*}
where $y=[y_{1},\dots,y_{m}]^{T} \in \R^{m\times p}$.
\end{algorithmic}
\end{algorithm}

Depending on the algorithm used for sampling control points, \cref{FunctionalCompressionAlg} yields different theoretical guarantees, which we now describe.
\begin{theorem}
\label{OrthogprojlemmaCurr}
     Suppose we have a discrete target functional of the form
    \begin{align*}
        \mu_{S}=\sum_{i=1}^{n}{\delta_{x_{i}}\alpha_{i}}  \in V^{*},\quad (x_{i},\alpha_{i}) \in \mathcal{X}\times \R^{p}
    \end{align*}
with associated dual vector-valued function
\begin{align*}
   v^{\alpha}(x) = \sum_{i=1}^{n}{k(x,x_{i})\alpha_{i}} \in V.
\end{align*}
Subsample $m$ distinct control points $\bm{c}=\{c_{i}\}_{i=1}^{m} \subset \{x_{i}\}_{i=1}^{n}$ and define the matrix of values 
    \begin{align*}
        Y_{c} = (v^{\alpha}(c_{1}),\dots,v^{\alpha}(c_{m}))^{T} \in \R^{m\times p},
    \end{align*}
    which is the evaluation of the dual vector-valued function on the control-point locations. 
    Computing weights
    \begin{align*}
        \beta = [\beta_{1},\dots,\beta_{m}]^{T}= K_{cc}^{-1}Y_{c} \in \R^{m\times p},
    \end{align*}
    yields an approximation 
    \begin{align*}
         \hat{\mu}_{S} =  \sum_{i=1}^{m}{\delta_{c_{i}}\beta_{i} } 
    \end{align*} 
    that satisfies
       \begin{align*}
      \norm{\mu_{S}-\hat{\mu}_{S}}_{V^{*}}^{2} \leq C\mathrm{tr}(K_{XX}-Q_{XX}).
    \end{align*}
\begin{enumerate}
\item Sampling control points from an $m$-DPP on the kernel matrix $K_{XX}$, yields the following in-expectation error bounds
   \begin{align*}
         \E_{\bm{c}}\norm{\mu_{S}-\hat{\mu}_{S} }_{V^{*}}^{2} \leq C(m+1)\sum_{i=m+1}^{n}{\lambda_{i}(K_{XX})}.
    \end{align*}
\item By sampling control points using the recursive RLS scheme of \cite{Musco}, we obtain similar theoretical guarantees (up to a factor). Fixing $\delta\in (0,\frac{1}{32})$, $S\in \N$, with probability $1-3\delta$,  we have $m\leq cS\log(S/\delta)$ and
   \begin{align*}
        \norm{\mu_{S}-\hat{\mu}_{S}}_{V^{*}}^{2} \leq \frac{Cn}{S}\sum_{i=S+1}^{n}{\lambda_{i}(K_{XX})}.
    \end{align*}
\end{enumerate}
\end{theorem}
We observe that the randomised projection scheme yields a strong error-decay bound, at least as fast as the rate of decay of eigenvalues of the kernel matrix. Furthermore, the proposed approximation scheme presented in \cref{FunctionalCompressionAlg} does not require any optimisation, and only requires a single sampling and projection operation. As we shall see in \cref{Numexpsect}, the resulting algorithm is rapid compared to existing compression methods. We also remark here that almost identical bounds to those proven for the recursive RLS scheme also hold when sampling control points with the DAC RLS sampler of \cite{DAC}, with slightly looser bounds on $m$. 

We may also relate the decay of error to the decay of eigenvalues of the kernel integral operator $\mathcal{K}_{p}$ itself, where $p$ denotes a sampling distribution for the delta centres $x_{i} \in \mathcal{X}$. This is summarised in the following corollary,
\begin{corollary}\label[corollary]{corrcontinuous}
Suppose that the Dirac delta centres $\{x_{i}\}_{i=1}^{n} \in \mathcal{X}$ of the target $\mu \in V^{*}$ are sampled from a continuous distribution with density $p$. Then, the rate of decay of the $m$-DPP and RLS approximations in \cref{OrthogprojlemmaCurr} are bounded in expectation and with high probability, respectively, by the rate of decay of the sum of decreasing ordered eigenvalues of $\mathcal{K}_{p}$, 
\begin{align*}
    f(m) = \sum_{i=m+1}^{\infty}{\lambda_{i}}
\end{align*}
In the special case where the input density $p$ is a mixture of $n$ Gaussians, each centred on some $c_{i}$ for $i \in [1,n]$, and the RKHS kernel is the Gaussian RBF, this yields the following exponential decay rate as a function of $m$
\begin{align*}
     f(m) = \mathcal{O}(m\exp(-\alpha m^{\frac{1}{d}})),\quad \alpha>0,
\end{align*}
where $\alpha$ is a distribution-dependent constant. 
\end{corollary}

In the case of currents and varifolds, a Gaussian mixture distribution over delta centres as in \cref{corrcontinuous}, encompasses the situation where the observed triangle centre and normal pairs $x_{i}$ of the underlying triangulation, are modelled as a noisy sample of true centres and normals $c_{i}$ from a noiseless ground truth shape.

We now discuss the application of 
\cref{FunctionalCompressionAlg} and \cref{OrthogprojlemmaCurr} to the compression of the currents and varifolds representations of shape.
\subsection{Application to compression of currents and varifolds}\label{CompressionofVarifolds}
Given a discretised surface $S \subset \R^{d}$ with $n=n_{T}$ triangles, we know that one can construct discrete currents and varifolds of the form
\begin{align*}
    \mu_{S} = \sum_{i=1}^{n}{\delta_{x_{i}}\alpha_{i}} \in V^{*} ,\quad (x,\alpha_{i})\in \mathcal{X}\times\R^{p},
\end{align*}
in order to perform downstream tasks requiring metric shape comparison, without correspondences. As before, $V$ is an RKHS induced by scalar diagonal kernel $K(x,y)=k(x,y)I_{p}$, with $k:\mathcal{X}\times \mathcal{X} \longrightarrow \R_{\geq 0}$. We can apply the approximation methods of the previous section in order to compress a given current or varifold in a sparse basis. It remains to make precise the forms of the underlying space $\mathcal{X}$, and the pairs $(x_{i},\alpha_{i}) \in \mathcal{X} \times \R^{p}$ for each representation.
\begin{enumerate}
    \item  For \textbf{currents}, we have that $\mathcal{X}=\R^{d}$, $p=d$ and for $i=1,\dots,n_{T},$
\begin{align*}
x_{i}:=\frac{1}{3}(v_{i1} + v_{i2} + v_{i3}),\quad \alpha_{i} := \frac{1}{2}(v_{i3}-v_{i2} ) \times ( v_{i2}-v_{i1}),
\end{align*}
so $x_{i},\alpha_{i}$ are the $i$'th triangle centre and (un-normalised) normal vector, respectively.

\item For \textbf{varifolds} the base space is changed to $\mathcal{X}=\R^{d} \times \mathbb{S}^{d-1}$, $p=1$, and for $i=1,\dots,n_{T},$,
\begin{align*}
    x_{i} \coloneqq \bigg(\frac{1}{3}(v_{i1} + v_{i2} + v_{i3}),\frac{\nu_{i}}{\norm{\nu_{i}}}\bigg),\quad \alpha_{i} \coloneqq \norm{\nu_{i}},\; \nu_{i}\coloneqq\frac{1}{2}(v_{i3}-v_{i2} ) \times ( v_{i2}-v_{i1})
\end{align*}
where $x_{i}$ is a tuple of $i$'th triangle centre and (normalised) normal vector, and $\alpha_{i}$ is the $i$'th triangle area.
\end{enumerate}
With these choices, applying the compression algorithm of the previous section will allow us to compress currents and varifolds within the same framework, while retaining strong theoretical error bounds. Indeed, \cref{OrthogprojlemmaCurr} and \cref{corrcontinuous} imply that for the common Gaussian kernel setting, one can obtain compressions of currents and varifolds associated to noisy shape samples with exponentially decaying compression error. As we shall see in \cref{Numexpsect}, the resulting compression algorithm is rapid in practice in addition to the theoretical guarantees.

\subsection{Choice of $m$}\label{choiceofm}
In the algorithms we have presented so far, the choice of the compression level $m$ is a crucial one.  In theory, one can set the eigenvalue bounds presented in \cref{OrthogprojlemmaCurr} to a pre-defined tolerance and choose $m$ accordingly. However, in practice, the cost of evaluating the upper bound is $\mathcal{O}(n^{3})$ due to the eigen-decomposition cost, which is undesirable. 


As a cheaper, practical alternative, we propose to use the preliminary trace bound of \cref{Corollary_trace} instead for this task. In particular, one can use the trace error of the Nystrom approximation as an indicator for the true compression error given $m$ control points. For shift invariant kernels, computation of the trace term simplifies as,
\begin{align*}
    tr(K_{XX}-Q_{XX}) = n - tr(K_{XC}K_{CC}^{-1}K_{CX}),
\end{align*}
which has a computational cost of at most $m^{3} + nm^{2}$, which is significantly less than $n^{3}$ of the eigen-decomposition.

One can choose $m$ by setting a pre-defined tolerance $\epsilon>0$ and increasing the number of RLS samples $m$, until the change in the trace bound is below this tolerance. This is a sensible choice, as we observe in practice (see \cref{Numexpsect}) that the decay of the true compression error closely matches that of the trace bound, which makes it a suitable proxy for the error. This is given as \cref{traceheuristic}. Note that this procedure is similar to the use of trace bounds for judging the quality of inducing points in sparse Gaussian process methods \cite{ratesofconvergesparsegp}.
\begin{algorithm}[H]
\caption{Trace bound heuristic for choosing $m$}\label{traceheuristic}
\begin{algorithmic}[1]
\State Fix error tolerance $\tau>0$, $e :=+\infty$ and $m=0$.
\While{$e > \tau$}{
\State $m \gets m+1$
\State RLS sample $m$ control points $\bm{c}=\{c_{i}\}_{i=1}^{m}$.
\State Set $e = n - tr(K_{XC}K_{CC}^{-1}K_{CX}) $   
\EndWhile
}
\end{algorithmic}
\end{algorithm}

\subsection{Application to control-point LDDMM matching}

Thus far, we have derived a randomised compression algorithm for large-scale currents and varifolds with theoretical guarantees on the quality of the approximation error. We now describe how the compressed measures can be used for the downstream task of nonlinear shape registration in the LDDMM framework, one of the areas in which these metrics are most commonly used. A typical workflow in the application of LDDMM to computational anatomy is to register two anatomical triangulated surfaces $T,S \subset \R^{d}$ (with $\lvert T\rvert,\lvert S\rvert$ triangles, respectively) via a minimal energy diffeomorphism $\varphi$. Such diffeomorphisms are typically induced as the flow $\varphi_{01}^{v}$ of time-dependent vector fields $v: [0,1] \times \R^{d} \longrightarrow \R^{d}$ such that $v(t,\cdot) \in V_{k}$, where $V_{k}$ is an RKHS of vector fields induced by a smooth positive definite spatial kernel $k$.

In practice, one often performs matching via a spatial control-point formulation of LDDMM, by minimising an objective $E: \R^{Pd}  \longrightarrow \R_{\geq 0}$, defined by the following system
\begin{align}\label{controlobjectiveintro}
E(\bm{\alpha}) := \frac{1}{2}\norm{v^{\bm{\alpha} }(0)}_{V_{k}}^{2} + \frac{\lambda}{2}\norm{\mu_{\varphi^{v^{\bm{\alpha} }}_{01}(T) } - \mu_{S} }_{V^{*}},\quad \bm{\alpha} =(\alpha_{1},\dots,\alpha_{P} ) \in \R^{Pd},
\end{align}
\begin{align}\label{controlpointfieldintro}
v^{\bm{\alpha} }(t,x) := \sum_{i=1}^{{P}}{k(x,c_{i}^{s}(t))\alpha_{i}(t)},
\end{align}
\begin{align}\label{controlpointHamiltonianintro}
\begin{cases}
 \partial_{t}c_{i}^{s}(t) = \sum_{j=1}^{P}{k(c_{i}^{s}(t),c_{j}^{s}(t))\alpha_{j}(t)},\quad c_{i}^{s}(0) = c_{i}^{s},\quad i=1,\dots,P\\
 \partial_{t}\alpha_{i}(t) = -\sum_{j=1}^{P}{\nabla_{1}k(c_{i}^{s}(t),c_{j}^{s}(t))\alpha_{i}(t)^{T}\alpha_{j}(t)},\quad \alpha_{i}(0) = \alpha_{i},\quad i=1,\dots,P
\end{cases}
\end{align}
The fields in \cref{controlpointfieldintro} are parametrised in terms of $P\ll \lvert T\rvert$ spatial deformation control points $\bm{c}^{s}:=\{c_{i}^{s}\}_{i=1}^{P} \in \R^{Pd}$ and initial `momenta' $\bm{\alpha}:=\{\alpha_{i}\}_{i=1}^{P} \in \R^{Pd}$ which completely determine the dynamics of the Hamiltonian system \cref{controlpointHamiltonianintro}. By minimising the objective \cref{controlobjectiveintro} as a function of the initial momenta, one seeks to find minimal energy time-dependent vector fields (satisfying \cref{controlpointHamiltonianintro}) deforming the template to target under the flow. The energy of the flow is penalised in the first term, and the discrepancy in the second term, measured in the dual norm on currents or varifolds. Note the spatial deformation control points $\bm{c}^{s}$ parametrising the diffeomorphism and compression control points ${\bm{c}}$ for measure compression in \cref{LDDMMcompalg} are distinct, and do not in general lie in the same space (e.g. for varifolds). We henceforth refer to the minimisation of \cref{controlobjectiveintro} as the \tit{uncompressed} matching problem.
 
 Computing the objective \cref{controlobjectiveintro} has two parts: flow computation and metric computation. One must also compute the gradients of these components for the optimisation procedure. It is common practice \cite{GlaunesCurrents} to approximate $\varphi(T)$, by computing the pushforward of the $n_{v}$ vertices of $T$, and forming a new triangulation from the deformed vertices with the \tit{same} mesh connectivity as the original template. This approximation has negligible error for finely sampled shapes. Therefore, at each step of the ODE solver for the flow, one computes a reduction at cost $\mathcal{O}(Pn_{v} )$, which may be implemented efficiently using KeOps, and is fast as one takes $P\ll n_{v} < \lvert T \rvert$ in the control point parametrisation of deformations. In practice, one also usually only requires simple first or second-order solvers with few evaluations for the flow. This leaves the registration discrepancy term of \cref{controlobjectiveintro}, which becomes a dominating cost with metric and gradient cost/memory still scaling as $\mathcal{O}(\lvert T\rvert\lvert S\rvert)$. This becomes especially noticeable in the group-wise setting when many shapes are being registered. 
 
 This is where our work may be applied to yield a speedup. Firstly, since the target $S$ does not change from iteration to iteration, one can compress $\mu_{S}$ using \cref{FunctionalCompressionAlg} as a form of pre-processing before running the registration algorithm. Then, the main limiting quadratic cost remaining is that of the metric terms involving $\mu_{\varphi^{\bm{\alpha}}(T)}$, which changes from iteration to iteration.

There are two options at this stage. The first is to compute the compression of $\mu_{\varphi(T)}$ using \cref{FunctionalCompressionAlg} at each iteration and make the metric comparison on the compressed measures. The second, is to compress $\mu_{T}$ itself offline, on a set of $m$ \textit{compression control points} $\bm{c}=\{c_{i}\}_{i=1}^{m}$ subsampled from the delta centres of $\mu_{T}$, and compute a push-forward of the compressed measure at each iteration as $\varphi_{\#}(\hat{\mu}_{T})$. With this pushforward one may compute the metric against the compressed target. 

There are trade-offs in either case. In the first case, {at each iteration}, we have to compute both the leverage score approximation and the orthogonal projection, with cost scaling as $\mathcal{O}(nm^{2})$ and $\mathcal{O}(m^{3})$, respectively. In the second case, one only has to compute the compression \textbf{once} offline and only push forward the compressed measure at every iteration. However, computing the push-forward of $\hat{\mu}_{T}$ under $\varphi$ requires computing the Jacobian of $\varphi$, which involves \cite{DURRLEMAN,VarCompression} solving another set of coupled ODEs. This can be expensive depending on the problem setup.

Instead, we observe if one were to compress $\mu_{T}$ on $\bm{c}=\{c_{i}\}_{i=1}^{m}$, the pushforward of $\hat{\mu}_{T}$ is supported on $\varphi(\bm{c}) \in \R^{md}$, the $m$ transported compression control points; see \cite{DURRLEMAN}, \cite{VarCompression} for this result on currents and varifolds, respectively. Motivated by this, instead of computing the measure push-forward $\varphi_{\#}\hat{\mu}_{T}$ explicitly per-iteration, we deform the full template $T$ under the spatial deformation control-point fields \cref{controlpointfieldintro} at cost $\mathcal{O}(Pn_{v})$, and subsequently perform the orthogonal projection of $\mu_{\varphi(T)}$ on the points $\varphi(\bm{c})$. That is, at each iteration, we compute the projection of the deformed template as $\mathcal{P}_{\varphi(\bm{c})}(\mu_{\varphi(T)})$ and use it for metric computation. The resulting algorithm avoids repeated RLS sampling and deformation Jacobian computations, but does require pushing forward the dense template. However, we note that the template pushforward operation has reduced cost anyway, in the spatial control-point LDDMM framework.

The proposed \textit{compressed matching} algorithm, making use of \cref{FunctionalCompressionAlg} is given below, which we use for our registration experiments in \cref{Numexpsect}. 

\begin{algorithm}[H]
\caption{Compressed matching algorithm.}\label[algorithm]{LDDMMcompalg}
\begin{algorithmic}[1]
\State Fix number of optimisation steps $N$, $i=0$, template $T$, target $S$, and compressed target measure $\hat{\mu}_{S}$.
\State Initialise spatial deformation control points $\bm{c}^{s} \in \R^{Pd}$ and associated initial momenta $\bm{\alpha} \in \R^{Pd}$.
\State  Apply RLS sampling from delta centres of $\mu_{T}$ to obtain compression control points ${\bm{c}}=\{c_{i}\}_{i=1}^{m}$. 

\While{$i<N$}{
\State Push forward template under the current estimated diffeomorphism,
\begin{align*}
\varphi^{\bm{\alpha}}:=\varphi^{v^{\bm{\alpha} }}_{01}
\end{align*}
as defined by \cref{controlpointfieldintro}, to obtain $\mu_{\varphi^{\bm{\alpha}}(T)}$
\State Project $\mu_{\varphi^{\bm{\alpha}}(T)}$ onto points $\varphi^{\bm{\alpha}}({\bm{c}})$, to obtain $\mathcal{P}_{\varphi^{\bm{\alpha}}(\bm{c})}(\mu_{\varphi^{\bm{\alpha}}(T)})$.
\State Compute `compressed objective'
\begin{align}\label{compobj}
E({\bm{\alpha}}) = \frac{1}{2}{\norm{v^{\bm{\alpha}}(0)}_{V_{k}} ^{2} } + \frac{\lambda}{2}\norm{\hat{\mu}_{S} - \mathcal{P}_{\varphi^{\bm{\alpha}}(\bm{c})}(\mu_{\varphi^{\bm{\alpha}}(T)})}_{V^{*}}^{2}
\end{align}

\State Compute gradients of objective \eqref{compobj} with respect to $\bm{\alpha}$ using either adjoint methods or automatic differentiation.
\State Update $\bm{\alpha}$ via gradient-based optimiser of choice and increment counter $i$.
\EndWhile}
\end{algorithmic}
\end{algorithm}

The per iteration cost of the objective computation in \cref{LDDMMcompalg} break down as follows; the ODE solve which is unchanged with cost $\mathcal{O}(Pn_{v})$, projection of $\mu_{\varphi(T)}$ onto points $\varphi(\bm{c})$ at cost $\mathcal{O}( m^{3})$, and metric computation between compressed measures at $\mathcal{O}(m^{2})$ (assuming $\mu_{S}$ is compressed to the same level as $T$). Overall, the effective cost of a forward pass in \cref{LDDMMcompalg} is therefore $\mathcal{O}( m^{3} + Pn_{v})$ incurred in compressing $\mu_{\varphi(T)}$ and solving forward under the control-point deformation. Therefore, provided $m^{3} < \lvert T\rvert^{2}$ this yields a strong saving in computational complexity for the forward pass compared to the uncompressed case, which has effective complexity $\mathcal{O}( \lvert T\rvert^{2} + Pn_{v})$. Furthermore, in practice, the resulting linear systems to be solved for the compression projection step are of size $m\ll \lvert T\rvert$, which are small enough to fit comfortably on a GPU. As such, the \textit{practical run-time} may be further reduced, as one may parallelise and accelerate the resulting linear solve. As a result, the compressed metric and its gradient computation are in practice much faster to compute than the uncompressed case. Much like the uncompressed case, all steps of \cref{LDDMMcompalg} may be further sped up via efficient libraries for kernel reduction, such as KeOps and using GPUs. Further work could involve additional approximations to the projection operation, exploiting sparsity, and using tools such as conjugate gradients.

\section{Proofs of main results}\label{CompressionofDiscreteMeasures}

We present in this section the proof of \cref{OrthogprojlemmaCurr}. We begin with a sparse control-point approximation of discrete functionals using the Nystrom approximation in RKHS.
\begin{lemma}\label[lemma]{fieldtraceapprox} 
    Given $v^{\alpha}$ of the form
    \begin{align*}
        v^{\alpha}(x) = \sum_{i=1}^{n}{k(x,x_{i})\alpha_{i}},\quad (x_{i},\alpha_{i}) \in \mathcal{X}\times \R^{p}
    \end{align*}
    $m$ distinct control points $\bm{c}=\{c_{i}\}_{i=1}^{m} \subset \{x_{i}\}_{i=1}^{n}$, and $\mu>0$, there is a control-point field of the form:
    \begin{align}\label{approxinv}
          v^{\bm{c},\mu} =  \sum_{i=1}^{m}{k(\cdot,c_{i})\beta_{i}^{\bm{c},\mu}},\quad \beta_{i}^{\bm{c},\mu} \in \R^{p},
    \end{align}
    and a constant $C_{1}$ (depending on $\mu$ and $v^{\alpha}$) such that the following holds
\begin{align}\label{triangle}
    \norm{v^{\alpha} - v^{\bm{c},\mu}}_{V }^{2} \leq C_{1}\mathrm{tr}(K_{XX}-Q_{XX}),\quad Q_{XX}:=K_{XC}K_{CC}^{-1}K_{CX}.
\end{align}
\end{lemma}
\begin{proof}
$ $\newline
The idea is to rewrite each component $v_{j}^{\alpha}$ as a solution to a KRR problem of the form \cref{KRRProblem}, and form its Nystrom approximation \cref{Nystromsol} in the control point subspace. 

Recall that given a field of the form $v^{\alpha}$, the $j$'th component can be written
\begin{align*}
    v_{j}^{\alpha}(x) = \sum_{i=1}^{n}{k(x,x_{i} )\alpha_{ij}  },\quad \alpha^{j} := (\alpha_{1j},\dots,\alpha_{nj} ),\quad y_{j}:=K_{XX}\alpha^{j} .
\end{align*}
In order to apply \cref{thmox}, we rewrite for arbitrary $\mu>0$
\begin{align*}
    \alpha^{j} = (K_{XX} + \mu I_{n})^{-1}(K_{XX} + \mu I_{n})\alpha^{j} =  (K_{XX} + \mu I_{n})^{-1}\big[(K_{XX} + \mu I_{n})K_{XX}^{-1}y_{j}\big],
\end{align*}
so that $v_{j}^{\alpha}$ is written as a kernel ridge regression solution with data
\begin{align*}
    \tilde{y}_{j} := (K_{XX} + \mu I_{n})K_{XX}^{-1}y_{j} = y_{j} + \mu\alpha^{j}.
\end{align*}
One can now form the Nystrom KRR solution approximating $v_{j}^{\alpha}$ as:
\begin{align*}
    v_{j}^{\bm{c},\mu}(x) =  \sum_{i=1}^{m}{k(x,c_{i})\beta_{ij}},\quad  (\beta_{ij})_{i=1}^{m} = (K_{CX}K_{XC} + \mu K_{CC})^{-1}K_{CX}\tilde{y}_{j},
\end{align*}
which gives the resulting trace bound by \cref{thmox}
\begin{align}\label{triangle1}
    \norm{v_{j}^{\alpha} - v^{\bm{c},\mu}_{j}}_{V_{k}}^{2} \leq  \frac{2\mathrm{tr}(K_{XX}-Q_{XX})\norm{\tilde{y_{j} }}^{2}}{\mu^{2}}.
\end{align}
By applying this bound component-wise, and applying \cref{normsplittinglemma}, this gives us the existence of a function of the form 
\begin{align*}
    v^{\bm{c},\mu}(x) = \sum_{i=1}^{m}{k(x,c_{i})\beta_{i}},\quad \beta_{i}=(\beta_{i1},\dots,\beta_{ip} ),\quad v^{\bm{c},\mu}=(v_{1}^{\bm{c},\mu},\dots,v_{p}^{\bm{c},\mu}),
\end{align*}
such that
\begin{align*}
    \norm{v^{\alpha} -   v^{\bm{c},\mu} }_{V}^{2} =\sum_{l=1}^{p}{\norm{v_{l}^{\alpha} - v^{\bm{c},\mu}_{l}}_{V_{k}}^{2}}\leq Ctr(K_{XX}-Q_{XX}),\quad C=\frac{2p}{\mu^{2}}\norm{\tilde{y}}^{2} ,
\end{align*}
so the error is controlled by the trace error of the Nystrom approximation.
\end{proof}
Applying the isometric duality mapping and the above result, this gives the following corollary
\begin{corollary}\label[corollary]{Corollary_trace}
    Suppose we have a discrete functional of the form
    \begin{align*}
        \sum_{i=1}^{n}{\delta_{x_{i}}\alpha_{i}}  \in V^{*},\quad (x_{i},\alpha_{i}) \in \mathcal{X}\times \R^{p},
    \end{align*}
    which we wish to compress. For any fixed $m<n$, there exists a discrete functional of the form
    \begin{align*}
        \sum_{i=1}^{m}{\delta_{c_{i}}\alpha^{c,\mu}_{i}} \in V^{*},
    \end{align*}
    depending on parameter $\mu>0$ such that the following bounds hold:
    \begin{align}\label{deltabound}
        \norm{\sum_{i=1}^{n}{\delta_{x_{i}}{\alpha}_{i} }  -  \sum_{i=1}^{m}{\delta_{c_{i}}\alpha_{i}^{c,\mu} } }_{V^{*}}^{2} \leq Ctr(K_{XX}-Q_{XX}).
    \end{align}
    where $C$ depends on $\mu$ and the functional to be compressed.
\end{corollary}
\begin{proof}

Setting $v^{\alpha}: = L^{-1}(\sum_{i=1}^{n}{\delta_{x_{i}}\alpha_{i}})$, choosing an approximating $v^{\bm{c},\mu}$ from the previous lemma, and by applying the isometric property of the duality mapping $L$  we have
\begin{align*}
 \norm{v^{\alpha} -   v^{\bm{c},\mu} }_{V}^{2} = \norm{L(v^{\alpha} -   v^{\bm{c},\mu}) }_{V^{*}}^{2} =  \norm{L(v^{\alpha}) -   L(v^{\bm{c},\mu}) }_{V^{*}}^{2} \\
 =\norm{L(\sum_{i=1}^{n}{k(x,x_{i})\alpha_{i}}) -   L(\sum_{i=1}^{m}{\alpha_{i}^{c,\mu} k(\cdot,c_{i})})}_{V^{*}}^{2} \\
 = \norm{ \sum_{i=1}^{n}{L(k(x,x_{i})\alpha_{i})} -   \sum_{i=1}^{m}{L(\alpha_{i}^{c,\mu} k(\cdot,c_{i}))}}_{V^{*}}^{2}\\ 
 =\norm{\sum_{i=1}^{n}{\delta_{x_{i}}{\alpha}_{i} }  -  \sum_{i=1}^{m}{\delta_{c_{i}}\alpha_{i}^{c,\mu} } }_{V^{*}}^{2}.
\end{align*}
Applying \cref{fieldtraceapprox} to the previous result yields
\begin{align*}
      \norm{\sum_{i=1}^{n}{\delta_{x_{i}}{\alpha}_{i} }  -  \sum_{i=1}^{m}{\delta_{c_{i}}\alpha_{i}^{c,\mu} } }_{V^{*}}^{2} \leq Ctr(K_{XX}-Q_{XX}).
\end{align*}
\end{proof}
So far, we have approximated a given discrete input functional sparsely, based on a subset of control points chosen from the original basis. We have shown that the error of this approximation is bounded by the trace error of the Nystrom approximation to the kernel matrix $K$, based on the chosen control points. Of course, the quality of this approximation is dependent on a good choice of control-point locations.  The following corollary shows that it suffices to randomly sample control points from a data-dependent input distribution in order to give strong theoretical error bounds on the approximation procedure.
\begin{corollary}\label[corollary]{CurrentExpectationbound}
    Suppose we sample $m$ control points $\bm{c}=\{c_{i}\}_{i=1}^{m}$ from an $m$-DPP as in \cref{eigenvaluethm}. Then, the approximation of \cref{Corollary_trace} satisfies the following in-expectation bound with respect to this distribution:
    \begin{align*}
         \E_{\bm{c}}\norm{\sum_{i=1}^{n}{\delta_{x_{i}}\alpha_{i} }  -  \sum_{i=1}^{m}{\delta_{c_{i}}\alpha^{c,\mu}_{i} } }_{V^{*}}^{2} \leq C(m+1)\sum_{i=m+1}^{n}{\lambda_{i}(K_{XX})},
    \end{align*}
    where $\lambda_{i}(K_{XX})$ are eigenvalues of the kernel-matrix evaluated on the discrete input points $x_{i}$.
\end{corollary}
\begin{proof}
Taking expectation in \cref{deltabound} with respect to an $m$-DPP distribution of control points, one obtains 
\begin{align*}
   \E_{\bm{c}}\norm{\sum_{i=1}^{n}{\delta_{x_{i}}\alpha_{i} }  -  \sum_{i=1}^{m}{\delta_{c_{i}}\alpha^{c,\mu}_{i} } }_{V^{*}}^{2} \leq  C\E_{\bm{c}}[tr(K_{XX}-Q_{XX})] .
\end{align*}
Applying \cref{eigenvaluethm} yields,
\begin{align*}
         \E_{\bm{c}}\norm{\sum_{i=1}^{n}{\delta_{x_{i}}\alpha_{i}  }  -  \sum_{i=1}^{m}{\delta_{c_{i}}\alpha^{c,\mu}_{i} } }_{V^{*}}^{2}  \leq  C\E_{\bm{c}}[tr(K_{XX}-Q_{XX})] \leq C(m+1)\sum_{i=m+1}^{n}{\lambda_{i}(K_{XX})},
\end{align*}
as required.
\end{proof}

    
Computing the weights for the approximation \cref{approxinv} however can be expensive, as it requires the one off cost of computation of a large-scale kernel matrix vector product $K_{XX}\alpha_{j}$ with cost $\mathcal{O}(n^{2})$, which is comparable to the original metric computation cost for currents and varifolds. For compressing currents as a form of pre-processing for downstream tasks such as matching and computing statistics, this is fine, as one can treat it as a one-off cost, before performing the task of interest. However, this can be tedious if one wishes to compress many currents or varifolds, or perform multiple compressions at different scales. 

One can make this compression algorithm more practical, while retaining the theoretical guarantees, by using the orthogonal projection onto the random subspace. This will be significantly cheaper to compute for a given target functional compared to the Nystrom approximation weights, while still having the same error bound in expectation. One may further make the approximation cheaper by using RLS sampling schemes for choosing control-point positions. This is the content of \cref{OrthogprojlemmaCurr}, which we now prove.

\begin{proof}[Proof of \cref{OrthogprojlemmaCurr}]
We denote by $\mathcal{P}_{c}$ the orthogonal projection into the subspace 
\begin{align*}
    M(\bm{c}) = \mathrm{span}\{ k(\cdot,c_{i})e_{j}:e_{j}\in\R^{p},\quad i=1,\dots,m\quad j=1,\dots,p\},
\end{align*}
spanned by the $m$ control points. We denote the orthogonal projection of $v^{\alpha}$ onto $M(\bm{c})$ as
\begin{align*}
    \mathcal{P}_{c}(v^{\alpha})=\mathcal{P}_{c}(\sum_{i=1}^{n}{k(x,x_{i})\alpha_{i}}) = \sum_{i=1}^{m }{k(x,c_{i})\beta_{i}}.
\end{align*}
By the orthogonal projection conditions applied to the RKHS $V$, one can show the explicit relation
\begin{align*}
   \beta =  [\beta_{1},\dots,\beta_{m}]^{T} = K_{CC}^{-1}Y_{C},
\end{align*}
for the projection weights \cite{connections}. Here we recall the notation,
\begin{align*}
    Y_{c} := (v^{\alpha}(c_{1}),\dots,v^{\alpha}(c_{m}))^{T} \in \R^{m\times p},
\end{align*}
a matrix of values, which is the evaluation of the dual vector-valued function on the control-point locations. 

Indeed, the orthogonal projection conditions yield
\begin{align*}
    \langle v^{\alpha} - \mathcal{P}_{c}(v^{\alpha}),k(\cdot,c_{k})e_{j}\rangle_{V} = 0,\quad k=1,\dots,m,\quad j=1,\dots,p.
\end{align*}
Expanding this out gives
\begin{align*}
    \langle \sum_{i=1}^{n}{k(x,x_{i})\alpha_{i}} - \sum_{i=1}^{m }{k(x,c_{i})\beta_{i}},k(\cdot,c_{k})e_{j}\rangle_{V}
   & =\sum_{i=1}^{n}{k(c_{k},x_{i})\alpha_{i}^{T}e_{j}} - \sum_{i=1}^{m }{k(c_{k},c_{i})\beta_{i}^{T}e_{j}}\\
    &= (\sum_{i=1}^{n}{k(c_{k},x_{i})\alpha_{i}^{T}}) e_{j}- (\sum_{i=1}^{m }{k(c_{k},c_{i})\beta_{i}^{T}})e_{j} \\
    &= Y_{k}^{T}e_{j} -  (\sum_{i=1}^{m }{k(c_{k},c_{i})\beta_{i}^{T}})e_{j} = 0,
\end{align*}
so that
\begin{align*}
    Y_{kj} = (\sum_{i=1}^{m }{k(c_{k},c_{i})\beta_{ij} }) \implies Y_{C}=K_{CC}\beta \implies \beta=K_{CC}^{-1}Y_{C}.
\end{align*}
We now prove the DPP sampling bound of \cref{OrthogprojlemmaCurr}. Using the optimality of the orthogonal projection in inner product spaces and \cref{fieldtraceapprox}, we have the following bound
\begin{align*}
\norm{\sum_{i=1}^{n}{k(x,x_{i})\alpha_{i}}  -   \sum_{i=1}^{m }{k(x,c_{i})\beta_{i}} }_{V }^{2}  = \norm{\sum_{i=1}^{n}{k(x,x_{i})\alpha_{i}}  -  \mathcal{P}_{c}(\sum_{i=1}^{n}{k(x,x_{i})\alpha_{i}}) }_{V }^{2} \\ \leq  \norm{v^{\alpha} - v^{\bm{c},\mu}}_{V }^{2} 
 \leq Ctr(K_{XX}-Q_{XX}),
\end{align*}
and therefore by taking expectations and using \cref{eigenvaluethm} gives,
\begin{align*}
\E_{\bm{c}}\norm{\sum_{i=1}^{n}{k(x,x_{i})\alpha_{i}}  -   \sum_{i=1}^{m }{k(x,c_{i})\beta_{i}} }_{V }^{2} \leq C\E_{\bm{c}}tr(K_{XX}-Q_{XX}) \leq C(m+1)\sum_{i=m+1}^{n}{\lambda_{i}(K_{XX})}.
\end{align*}
Applying the isometry property of the Riesz map once again yields,
\begin{align*}
\E_{\bm{c}}\norm{\sum_{i=1}^{n}{\delta_{x_{i}}\alpha_{i} }  -  \sum_{i=1}^{m}{\delta_{c_{i}}\beta_{i} } }_{V^{*}}^{2} &= \E_{\bm{c}}\norm{\sum_{i=1}^{n}{k(x,x_{i})\alpha_{i}}  -   \sum_{i=1}^{m }{k(x,c_{i})\beta_{i}} }_{V }^{2}\\& \leq C(m+1)\sum_{i=m+1}^{n}{\lambda_{i}(K_{XX})}.
    \end{align*}

We now prove the RLS sampling bound. For this proof, we note that the following holds \cite{connections} for symmetric positive-definite matrices $A \in \R^{n\times n}$:
    \begin{align}\label{tracespectral}
        tr(A) \leq n\norm{A}_{2},
    \end{align}
    where the norm on the right-hand side is the spectral norm.
    
    Suppose we fix $\delta\in (0,\frac{1}{32})$, $S\in \N$ and sample $m$ control points $\bm{c}=\{c_{i}\}_{i=1}^{m}$ from the recursive RLS sampler of \cite{Musco}. From the proof of \cref{OrthogprojlemmaCurr}, we recall that the orthogonal projection approximation based on the given control point set satisfies
    \begin{align*}
        \norm{\sum_{i=1}^{n}{\delta_{x_{i}}\alpha_{i} }  -  \sum_{i=1}^{m}{\delta_{c_{i}}\beta_{i} } }_{V^{*}}^{2} \leq Ctr(K_{XX}-Q_{XX}).
    \end{align*}
    The inequality \eqref{tracespectral} implies that
    \begin{align}
        tr(K_{XX}-Q_{XX}) \leq n\norm{K_{XX}-Q_{XX}}_{2},
    \end{align}
    and therefore,
     \begin{align*}
        \norm{\sum_{i=1}^{n}{\delta_{x_{i}}\alpha_{i} }  -  \sum_{i=1}^{m}{\delta_{c_{i}}\beta_{i} } }_{V^{*}}^{2} \leq Cn\norm{K_{XX}-Q_{XX}}_{2}.
    \end{align*}
    By applying the inequality of \cref{RLSIneq} for recursive RLS sampling to the right-hand side, we know with probability $1-3\delta$ that $m<cS\log(S/\delta)$ and
    \begin{align*}
        \norm{\sum_{i=1}^{n}{\delta_{x_{i}}\alpha_{i} }  -  \sum_{i=1}^{m}{\delta_{c_{i}}\beta_{i} } }_{V^{*}}^{2} \leq \frac{Cn}{S}\sum_{i=S+1}^{n}{\lambda_{i}(K_{XX})}.
    \end{align*}
    This concludes the proof.   
\end{proof}

Finally, we give the proof of \cref{corrcontinuous}, which now follows from the eigenvalue bounds by taking expectation and union bounds.
\begin{proof}[Proof of \cref{corrcontinuous}]
   We have observed so far that both the $m$-DPP and RLS sampling cases are upper bounded up to a multiplicative constant by the term
    \begin{align*}
        \sum_{i=m+1}^{n}{\lambda_{i}(K_{XX})},
    \end{align*}
    as a function of $m$. By the identity \cref{eigdecay}, we have
    \begin{align*}
        \E_{X}\bigg [\sum_{i=m+1}^{n}{\lambda_{i}(K_{XX})}\bigg] \leq n \sum_{i=m+1}^{\infty}{\lambda_{i}}.
    \end{align*}
    Markov's inequality may also be applied to the above to yield a probability $1-\delta$ bound of,
    \begin{align*}
        \sum_{i=m+1}^{n}{\lambda_{i}(K_{XX})} \leq \frac{n}{\delta} \sum_{i=m+1}^{\infty}{\lambda_{i}}.
    \end{align*}
    Taking expectation/union bound and inserting the above bounds into the $m$-DPP and RLS bounds of \cref{OrthogprojlemmaCurr} respectively, gives the first part of \cref{corrcontinuous}.
    
    For the proof of the second part of \cref{corrcontinuous}, it is known by lemma 11 of \cite{JounralSParseVIGP}, that the eigenvalue decay of $\mathcal{K}_{p}$ where $p$ is a mixture of Gaussians can be upper bounded by that of a single Gaussian $p$ with sufficiently large variance. The rate of decay of eigenvalues of $\mathcal{K}_{p}$ for a single Gaussian $p$ is already known to be \cref{theoreticaleigdecay}, which yields the desired exponential decays, by inserting the estimates into the first part of \cref{corrcontinuous}.
\end{proof}

\section{Numerical experiments}\label{Numexpsect}

We now illustrate the performance of our algorithm for the compression of the currents and varifolds representation of real-world shapes taken from modern geometry processing datasets. We explore the following main strengths of our algorithm for compression of geometric measures: {Fast compression times}, {rapid error decay} supported by theoretical bounds, and applicability to large-scale registration problems. We illustrate both the true error and theoretical rates when computationally feasible.

In all subsequent experiments in this section, we use the algorithm of \cite{DAC} for RLS approximation as opposed to the recursive algorithm of \cite{Musco}, for which we prove the second bound of \cref{OrthogprojlemmaCurr}. Both methods yield similar theoretical upper bounds on the Nystrom error, with the \cite{Musco} being tighter. However, despite the tighter theoretical bounds of \cite{Musco}, in practice, the algorithm of \cite{DAC} gives similar error decay while being significantly faster in approximating RLS scores. Therefore, this is our preferred choice of RLS approximation to use in \cref{FunctionalCompressionAlg}. 

All experiments are implemented in PyTorch \cite{Torch}, and run on a Tesla T4 GPU with 16 GB of RAM. We also make use of the PyKeops framework \cite{Keops} to further speed up all kernel reduction operations. We make the code used to generate the results of this section publicly available \footnote{\url{https://github.com/tonyshardlow/GeometricMeasureCompression}}. 
\subsection{Error decay}\label{errordecaysec}
In this section, we study the rate of decay of the compression error as a function of $m$, the number of control points sampled for compression. We use the surface data illustrated in \cref{animalsfig} for our first experiments in this section. This data originates from \cite{ANIMdataset}.
\intextsep = 10pt

\begin{figure}[H]
   \begin{subfigure}[b]{1\textwidth}
\includegraphics[width=.4\textwidth,height=1.8in]{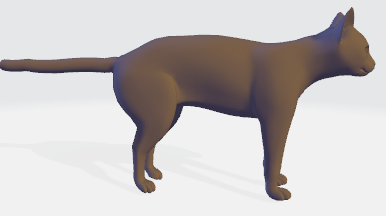} 
\hfill
          \includegraphics[width=.21\textwidth]{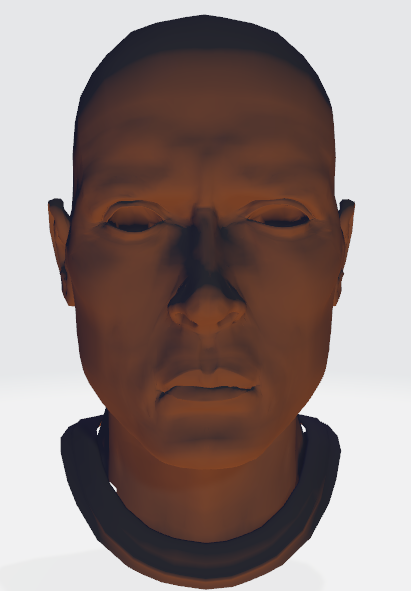} 
          \hfill
     \includegraphics[width=.25\textwidth]{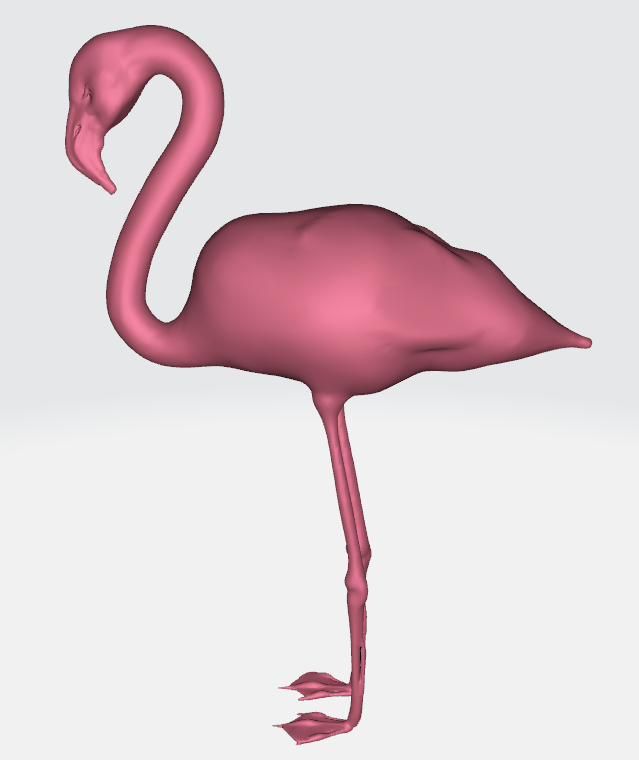} 
          \end{subfigure}
   \caption{\textbf{Left}: Cat (14410 triangles)\quad\textbf{Middle}: Head (31620 triangles)\quad \textbf{Right}: Flamingo (52895 triangles)}\label{animalsfig}
\end{figure}
The data are centred and scaled, so that the cat surface lies in a box of size $1.3\times 3.3 \times 7.1$, the head surface lies in a box of size $3.8\times 5.3 \times 4.0$, and the flamingo surface lies in a box of size $1.5\times 5.3 \times 3.8$. 

For each test surface, we run \cref{FunctionalCompressionAlg} for both currents and varifolds representations, and plot the true (relative) squared error $E=\norm{\mu-\hat{\mu}_{m}}_{W^{*}}^{2}$ of the compression as a function of $m$, where $\mu$ is the target, and $\hat{\mu}_{m}$ denotes an approximation formed with $m$ delta-centres. We also plot for comparison, the error curve for uniform sampling. By uniform sampling, we mean the approximation to $\mu$ obtained by uniformly sampling (without replacement) a subset of Dirac delta centres from $\mu$ and orthogonally projecting onto the subspace spanned by the associated deltas. Finally, the trace bound on the squared error derived in \cref{Corollary_trace} is also plotted, with rescaled numerically tighter constants. Note that we do not plot the curves for eigenvalue bounds, due to the prohibitive cost of computing the eigendecomposition as $m\to n$. 

For the currents experiments, we choose the Gaussian kernel $k(x,y)=\exp(-\frac{\norm{x-y}^{2}}{2\sigma^{2}})$ for $K_{p}$, and set the spatial kernel length-scale equal to $\sigma=0.5$ for all three test cases. For each test shape, one may compute the ratio between $\sigma$ and the shape's diameter to obtain $\nu =\sigma/d$, which contextualises the RKHS scale as a fraction of the data scale. This yields $\nu = 0.063$ for the cat surface, $\nu = 0.065$ for the head surface and $\nu=0.075$ for the flamingo surface. For the varifolds experiments, we choose the Gaussian kernel  $k(x,y)=\exp(-\frac{\norm{x-y}^{2}}{2\sigma^{2}})$ for $K_{p}$, and the spherical Gaussian kernel $k(s,r) = \exp(-\frac{(2-2\langle s,r\rangle)}{2\sigma_{s}^{2}})$ for $K_{s}$. Here, we choose $\sigma_{s}=0.5$ and $\sigma_{p}\in \{0.3,0.5,0.25\}$, respectively, for the cat ($\nu = 0.038$), head ($\nu = 0.065$) and flamingo ($\nu = 0.037 $) test cases. 

All results are shown in \cref{plots1}.
\begin{figure}[H]
\centering
   \begin{subfigure}[b]{.6\textwidth}
         \centering
         \includegraphics[width=.49\textwidth]{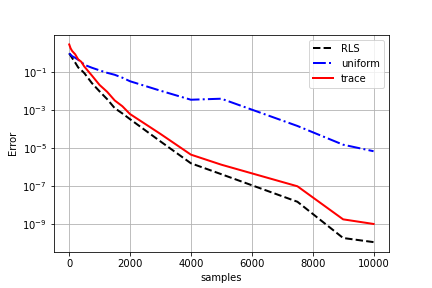}
            \includegraphics[width=.49\textwidth]{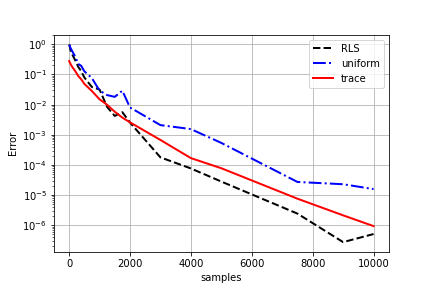}
 
         \includegraphics[width=.49\textwidth]{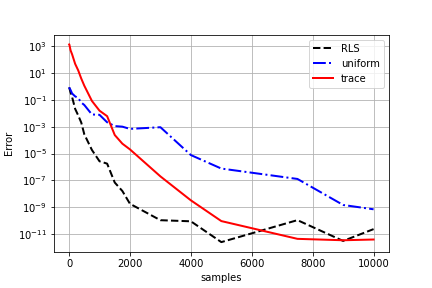}
         \includegraphics[width=.49\textwidth]{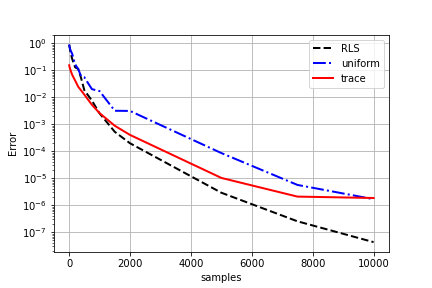}

         \includegraphics[width=.49\textwidth]{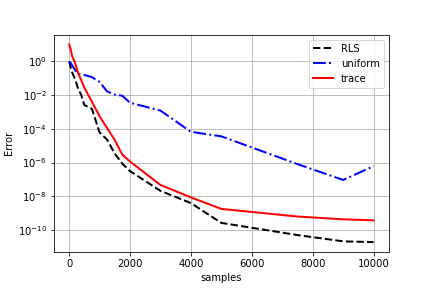}
         \includegraphics[width=.49\textwidth]{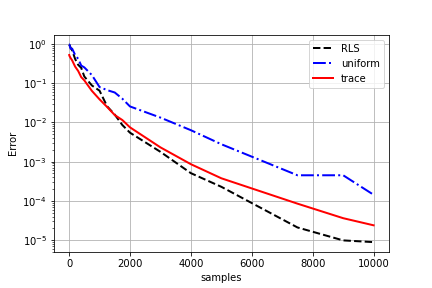}
     \end{subfigure}
     \caption{Numerical curves comparing (relative) error decay of RLS compression (black) to theory bound (red) and uniformly sampled compression (blue), on cat (top), head (middle) and flamingo (bottom) surfaces. \textbf{Left}: Error curves for currents. \textbf{Right}: Error curves for varifolds.  }\label{plots1}
\end{figure}

We observe in all cases that the decay of error of the compressed approximation is rapid, and one can take $m \ll n$ for a good quality of approximation, across all the example cases. One also observes that the derived trace bounds (in red) have decay rate that generally matches that of the true squared error of our algorithm. Finally, we observe that RLS sampling consistently outperforms uniform sampling in all cases. 

The RLS sampling compression also tends to produce better \tit{quality} samples (in terms of squared error) than uniform sampling. In general, uniform sampling will tend to have more points sampled in regions which are oversampled to begin with, thus only representing the underlying surface well only in densely sampled regions. On the other hand, RLS sampling measures the local `importance' of triangles through the ridge leverage scores and samples them accordingly. Thus, RLS sampling tends to produce more `diverse' samples \cite{Musco} and will produce samples that are well spread out, independent of the initial sampling density. 

We further illustrate qualitatively in \cref{tangentpatches}, samples from the compression algorithm for each of the three test shapes, in the varifold case. We do so by visualising the resulting `tangent patches' centred on control points and normal vectors as in \cite{VarCompression}. 
\raggedbottom
\begin{figure}[H]
   \begin{subfigure}[H]{1\textwidth}
   \hspace{-20pt}
         \includegraphics[width=.5  \textwidth,height=1.8in]{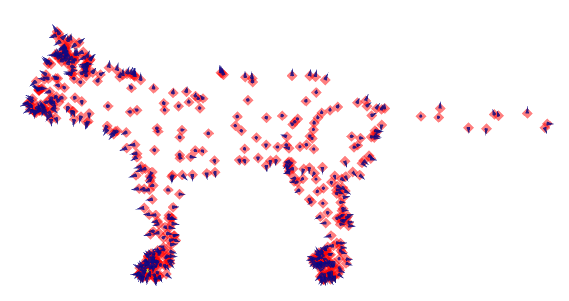}
         \hspace{-10pt}
         \includegraphics[width=.28\textwidth]{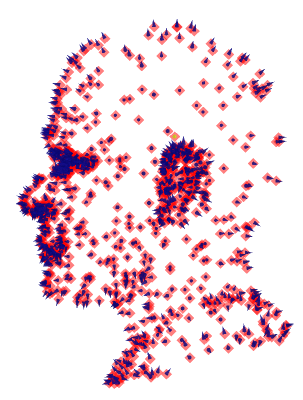}\includegraphics[width=.26\textwidth]{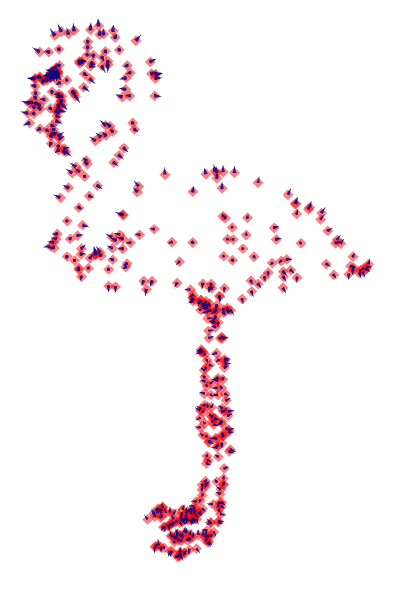}
     \end{subfigure}
     \caption{Visualisation of compressed varifolds for cat (left), head (middle) and flamingo (right) test surfaces, at resolution $m=500$ triangles (over $97\%$ compression ratio). Dark blue cones denote normal directions in the compression, and red patches indicate the location of sampled control points.  }\label{tangentpatches}
\end{figure}

Finally, we note that in many geometry processing applications, one often does not require an extremely small error in order to perform downstream tasks with the compressed representation. In many situations, the data itself is often acquired in a noisy way and can contain many local variations that are not relevant in describing the global geometry. As such, the above curves suggest one can practically choose $m\ll n$ and obtain an acceptable error for the down-line tasks which we perform with the compressions. For example, \cite{DURRLEMAN} suggests a heuristic of $\tau = 5\%$ relative error cut-off for compression of currents in the orthogonal matching pursuit algorithm. 
\vspace{10pt}

\subsection{Matching quality}\label{matchqualitysec}
We now illustrate the effectiveness of the compressed measures for nonlinear shape registration in the LDDMM framework. In particular, we demonstrate that one can obtain a comparable quality of registration to the full (uncompressed) matching problem, when using the compressed matching algorithm \cref{LDDMMcompalg} with $m\ll n$, even when only $1-2\%$ of the underlying triangles are used in the compression of deformed template and target. We shall also demonstrate that using the compressed representation gives a large computational saving for the registration algorithm in terms of run-time.  

We demonstrate on two extremely densely sampled shapes taken from Thingi10k \cite{Thingi10K}, a modern geometry processing dataset. The first, a super high-resolution version of the Stanford bunny with \textbf{$259,898$} triangles, and the second, a high-resolution brain surface with \textbf{$350,328$} triangles. The bunny is centred and scaled to lie in a box centred at the origin of size $4.0\times 4.1 \times 3.4$, and the brain surface to a box of size $3.5\times 3.5 \times 3.6$. In both cases, we compare the matching quality and matching time taken for the uncompressed and compressed matching (using \cref{LDDMMcompalg}) problems. 

In both matching problems, we choose varifolds as our choice of discrepancy. For the varifolds kernels we choose $K_{p}$ to be Gaussian with length-scale $\sigma_{p}=0.2$, and $K_{s}$ to be spherical Gaussian of length-scale $\sigma_{s}=0.3$. We shall compare the final quality of the obtained registrations in the Hausdorff metric, rather than in the varifolds metric, as it is independent of the optimisation objective and gives a measure of alignment of both surfaces without correspondence. Recall that the Hausdorff metric between two sets $A,B \subset \R^{d}$ is defined as follows,
\raggedbottom
\begin{align*}
\vspace{-50pt}
    d_{H}(A,B) := \mathrm{max}\bigg(  \underset{a \in A}{\sup }{\ d(a,B)},  \underset{b \in B}{\sup }{\ d(b,A)}\bigg),\quad d(x,A) := \underset{a \in A}{\inf}{\norm{x-a}_{2}}.
\end{align*}
\raggedbottom
The quality of registration is compared by computing the Hausdorff metric between the target and deformed template. 

The template is in both cases taken to be a sphere mesh lying in a box centred at the origin of size $3.5\times 3.5 \times 3.5$, with the same resolution as the target surface. For the spatial kernel of the RKHS of deformation vector fields, we fix a sum of $4$ Gaussian kernels of decreasing length-scales $\sigma_{i} \in \{1.0,0.5,0.2,0.1\}$.  For all experiments, the diffeomorphism and push-forward are computed via a forward Euler scheme with $10$ time-steps. The kernel reductions and gradient computations for the diffeomorphism are performed using Keops \cite{Keops} and automatic differentiation. Optimisation is performed via an LBFGS routine run for $500$ iterations. We choose $P$, the number of spatial deformation control points of \cref{LDDMMcompalg} parametrising the diffeomorphism, to be the same as the number $m$ of varifold compression points. Spatial deformation control points are sampled from the spherical template uniformly and subsequently fixed for both compressed and uncompressed matching problems.  

We begin with the comparison for the bunny surface. We compress the target offline and the deformed template at each iteration, down to $7500$ triangles each, which is a compression ratio of \textbf{$97\%$}. Matching results are shown below in \cref{matchingrabbitcomp}.
\begin{figure}[H]
 \hspace{50pt}\begin{subfigure}{0.35\textwidth}
     \includegraphics[width=\textwidth]{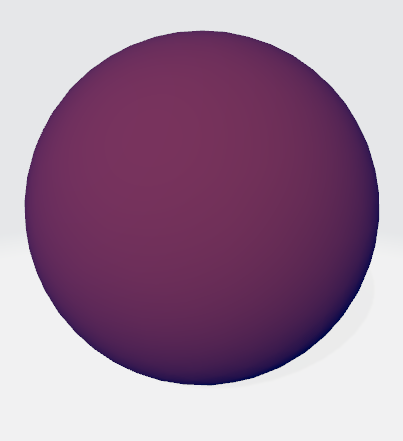}
 \end{subfigure}
 \hspace{20pt}
 \begin{subfigure}{0.35\textwidth}
     \includegraphics[width=\textwidth]{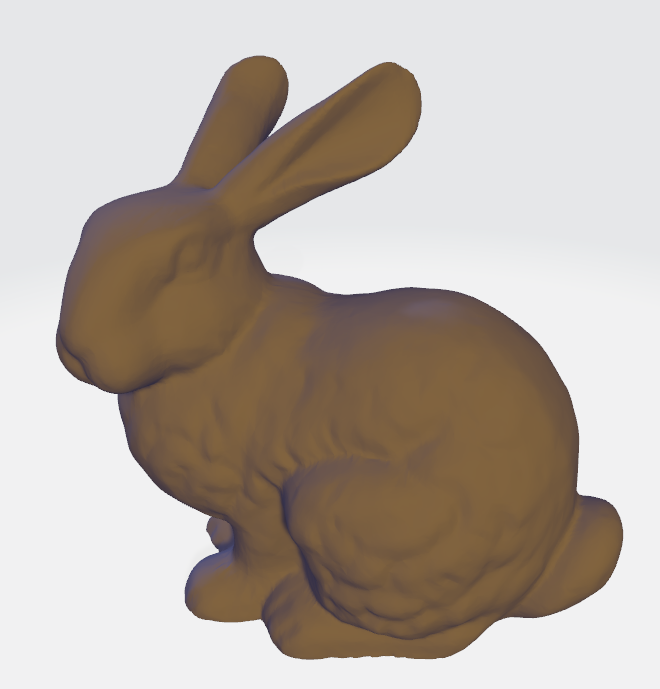}
 \end{subfigure}
 
 \medskip
 \hspace{50pt}\begin{subfigure}{0.35\textwidth}
     \includegraphics[width=\textwidth]{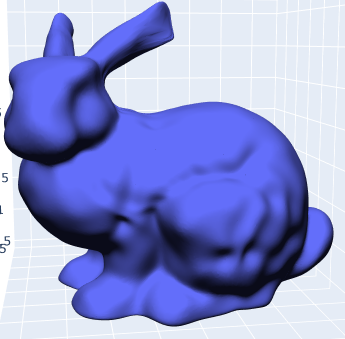}
 \end{subfigure}
 \hspace{20pt}
 \begin{subfigure}{0.35\textwidth}
     \includegraphics[width=\textwidth]{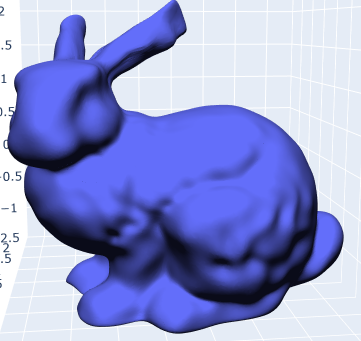}
 \end{subfigure}
 \caption{\textbf{Top left}: spherical template. \textbf{Top right}: target mesh. \textbf{Bottom left}: Matching with full metrics taking $1$ hour and $41$ minutes, and Hausdorff metric error $d_{H} = 0.026$. \textbf{Bottom right}: Matching with $97\%$ compression of template and target taking only $14$ minutes, and Hausdorff metric error $d_{H} = 0.030$. }\label{matchingrabbitcomp}
\end{figure}
We observe in \cref{matchingrabbitcomp} that the matching quality is almost identical in $d_{H}$ for compressed and uncompressed registrations, and in some regions, better than the full matching. As one expects, the compressed matching algorithm yields a significant speed-up of $7$ times over the uncompressed version, reducing overall matching time from $\textbf{5918}$s to $\textbf{860}$s.

In the second example, we consider the analogous comparison for the brain surface. We compress the target offline, and the deformed template at each iteration down to $5000$ triangles each, which yields a compression ratio of \textbf{$99\%$}. The results are shown below in \cref{matchingbraincmop}.

\begin{figure}[H]
 \hspace{50pt}\begin{subfigure}{0.3\textwidth}
     \includegraphics[width=\textwidth]{Images/Matching/sphere.png}
 \end{subfigure}
 \hspace{20pt}
 \begin{subfigure}{0.32\textwidth}
     \includegraphics[width=\textwidth]{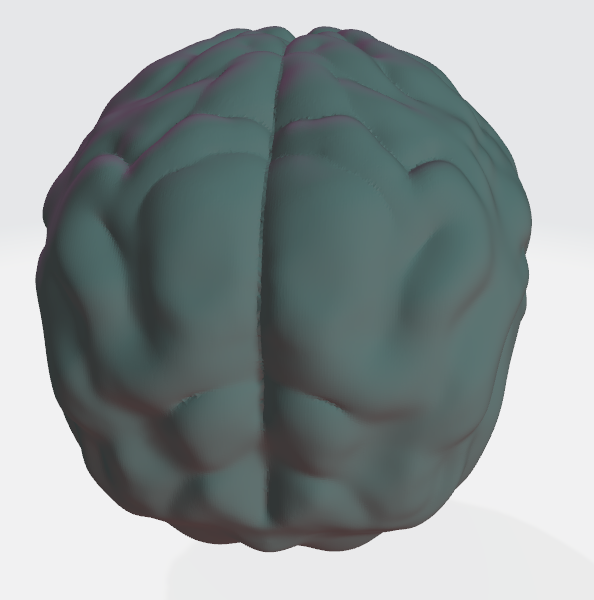}
 \end{subfigure}
 
 \medskip
 \hspace{50pt}\begin{subfigure}{0.3\textwidth}
     \includegraphics[width=\textwidth]{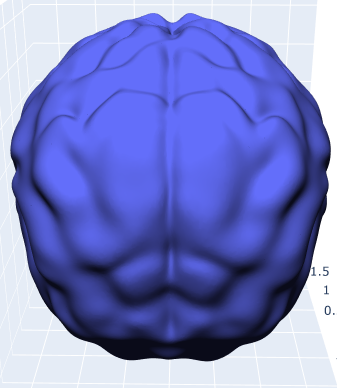}
 \end{subfigure}
 \hspace{20pt}
 \begin{subfigure}{0.33\textwidth}
     \includegraphics[width=\textwidth]{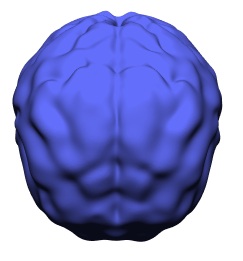}
 \end{subfigure}
 \caption{\textbf{Top left}: spherical template. \textbf{Top right}: target mesh.  \textbf{Bottom left}: matching without compression, taking $1$ hour and $49$ minutes, and Hausdorff metric error $d_{H} = 0.005$. \textbf{Bottom right}: the same example matching but with $99\%$ compression taking only $11$ minutes, and Hausdorff metric error $d_{H} = 0.007$.
  }\label{matchingbraincmop}
\end{figure}

Once again, this yields sizeable runtime savings, with a factor of $10$ times speed-up over the uncompressed matching problem, with overall matching time reduced from $\textbf{6540}$s to $\textbf{660}$s, while still maintaining a similiar $d_{H}$ error.

\subsection{Compression speed}\label{compressionspeedsec}
Finally, we compare the runtime of our compression algorithm to existing compression methods for currents \cite{DURRLEMAN} and varifolds \cite{VarCompression}. We demonstrate in this section that one can compress currents and varifolds to a fixed size using our algorithm, in a fraction of the run-time of the existing algorithms, while retaining fast decay of the compression error in dual metric as a function of $m$. This makes the RLS compression particularly suited to repeated/per-iteration compression where one wishes to compress shapes routinely as part of a general geometric learning algorithm. Furthermore, in situations where many compressed measures need to be computed, such as group-wise large-scale shape modelling, our algorithm has a distinct advantage due to rapid compression times and error decay. 

First, we illustrate how the run-time (in seconds) and compression error of \cref{FunctionalCompressionAlg} compare to the existing method \cite{DURRLEMAN} for \textbf{currents} compression, which greedily adds compression control points one by one based on the true approximation error. We demonstrate this on the cat and head example shapes from \cref{errordecaysec}. For both examples, we set the spatial kernel $K_{p}$ to be Gaussian with length-scale $\sigma=0.5$. Results are shown in \cref{greedyfig}.

\begin{figure}[H]
      \begin{subfigure}[b]{1\textwidth}
         \includegraphics[width=.45\textwidth]{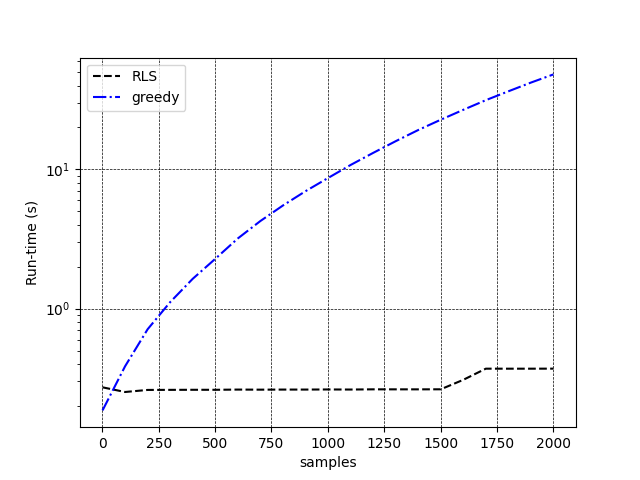}
         \includegraphics[width=.49\textwidth]{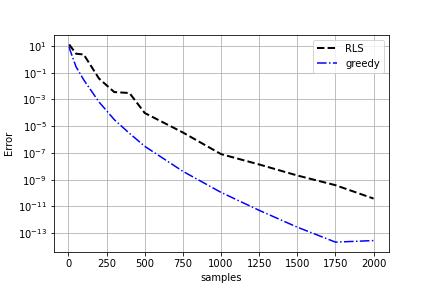}
           \end{subfigure}
           
      \begin{subfigure}[b]{1\textwidth}
      
         \includegraphics[width=.45\textwidth]{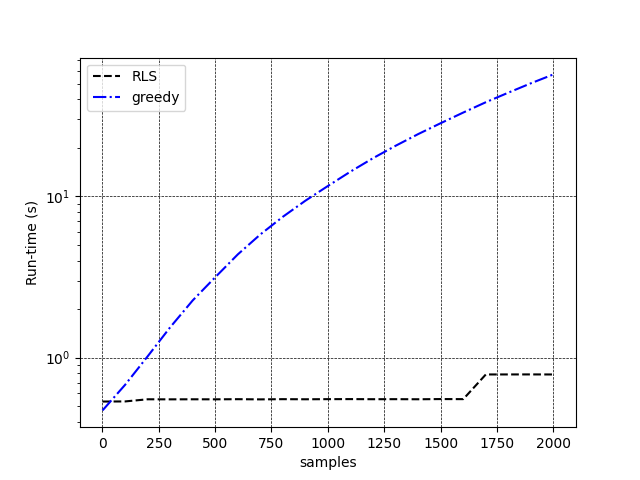}
         \includegraphics[width=.49\textwidth]{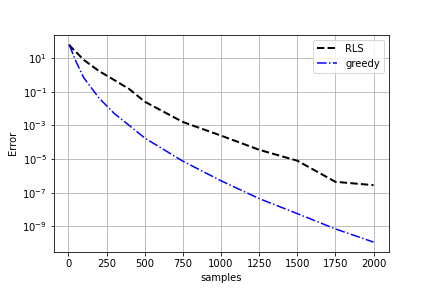}
     \end{subfigure}
     \caption{Run-time in seconds (left column) and error comparison (right column) of RLS method vs Greedy method \cite{DURRLEMAN} for compression of \textbf{currents} on example surface `cat' (top row) and `head' (bottom row). } \label{greedyfig}
\end{figure}

In \cref{greedyfig}, we observe that our proposed compression algorithm is up to $100\times$ faster in this range than the greedy method. The run-time of \cite{DURRLEMAN} grows rapidly as a function of sample size $m$, whereas our method grows extremely slowly. In terms of error, we observe that the \tit{rate} of decay is similarly fast for both methods, although the greedy method tends to require generally fewer samples to achieve a similiar error.


Next, we illustrate how the run-time and compression error of \cref{FunctionalCompressionAlg} compares to the existing \cite{VarCompression} state of the art for \textbf{varifolds} compression via optimisation. These comparisons are made on the bunny and brain examples of \cref{matchqualitysec}. We are able to make the comparison here on much larger examples, as the RLS sampling and varifold quantisation \cite{VarCompression} both have run-times that are much faster than \cite{DURRLEMAN}. For both examples, we set $K_{p}$ to be the Gaussian kernel with length-scale $\sigma=0.5$, and $K_{s}$ to be spherical Gaussian with $\sigma_{s}=0.5$. For the algorithm of \cite{VarCompression}, for each sample size $m$, we initialise the optimisation with delta centres uniformly randomly sampled from those of the ground truth varifold. This is the same initialisation used in \cite{VarCompression}. Results are shown in \cref{comparisonofvarifoldcompressiontime}.

\begin{figure}[H]
      \begin{subfigure}[b]{1\textwidth}
        \includegraphics[width=.4\textwidth]{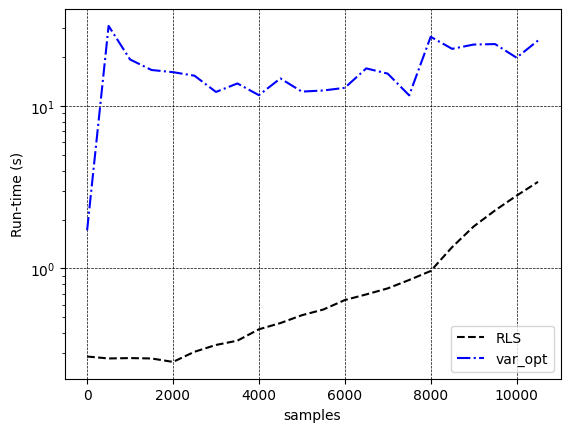}
         \includegraphics[width=.49\textwidth]{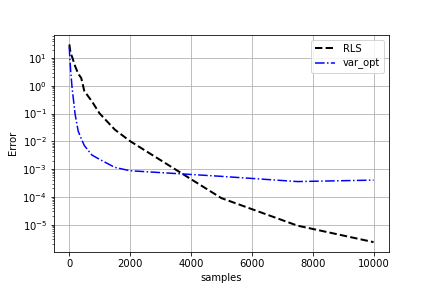}
     \end{subfigure}
  \begin{subfigure}[b]{1\textwidth}
       \includegraphics[width=.4\textwidth]{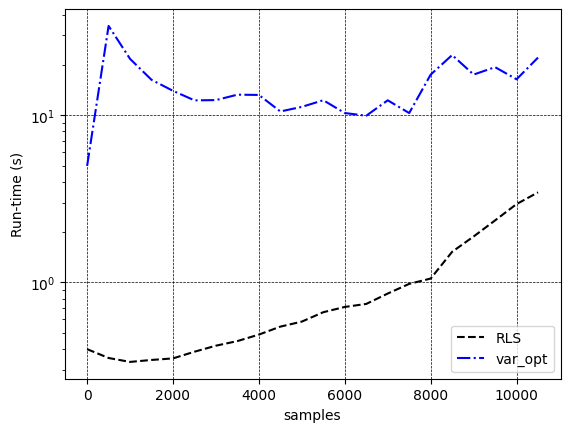}
         \includegraphics[width=.49\textwidth]{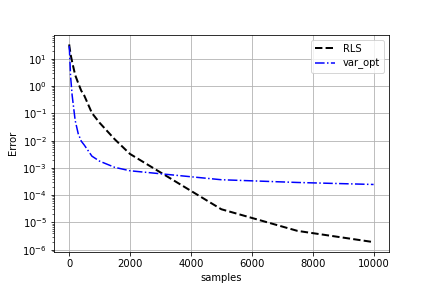}
     \end{subfigure}
     \caption{Run-time in seconds (left column) and error comparison (right column)  of RLS method vs optimisation method \cite{VarCompression} for compression of \textbf{varifolds}, on example surface `bunny' (top row) and `brain' (bottom row). }\label{comparisonofvarifoldcompressiontime}
\end{figure}

For varifolds, RLS sampling outperforms the optimisation-based approach in run-time, and tends to be up to $40\times$ faster in this range, demonstrating the benefits of sampling vs explicit optimisation. On the other hand, the optimisation-based compression scheme exhibits faster initial error decay. However, we observe that the method of \cite{VarCompression} can become stuck in local minima as the compression size $m$ increases, while the RLS-based approach continues to decrease the error as $m$ increases, due to the theoretical guarantees on our method's convergence. Indeed, the method of \cite{VarCompression} does not provide any convergence guarantees on the approximation error, and relies on a non-convex optimisation which is sensitive to initialisation. For such methods, one may require multiple runs with different initialisations, and iterative/nested continuation-type schemes where the output of compression at a coarse level is used as initialisation for a finer level. An advantage of our approach in comparison is that we have strong error bounds with high probability over draws of the RLS distribution, resulting in a more robust compression scheme.

\section{Conclusion}\label{concplusfut}
In this work, we have derived an algorithm for the compression of large-scale currents and varifolds using ridge leverage score sampling and the Nystrom approximation in RKHS. We have derived convergence bounds and rates of error decay on this compression algorithm as a function of the chosen kernels and associated parameters. The rapidly decaying error bounds, scalability, and fast run-times of the algorithm make it well-suited to routine use as well as pre-processing of measure-based shape representations. Numerous practical examples highlight the strengths of this framework for compression and accelerating registration, especially in the large-scale setting. We have also demonstrated the benefits of our method over existing compression techniques, which rely on greedy iterative schemes and non-convex optimisation procedures, both of which become slow in the large-scale setting. 


We also leave as future work the task of extending/modifying our algorithm to compression of higher-order geometric measure representations such as normal cycles \cite{NormalCycles}; a generalisation of currents, which is provably sensitive to curvature information and boundaries, in discrete and continuous shapes. Such properties make it an attractive choice of shape metric, especially for nonlinear shape registration frameworks, such as LDDMM. However, the normal cycles representation comes at a significantly increased computational cost, making it unattractive in large-scale applications. A compression algorithm adapted to the normal cycle representation would finally allow the routine use of such higher-order measures on large-scale geometry processing applications.

\newpage
\appendix


\section{Norm splitting}\label{normlemma}

We now show that the computation of the squared $V$ norm for certain types of vector fields we are interested in splits into a sum of squared $V_{k}$ norms over each dimension.

\begin{lemma}\label[lemma]{normsplittinglemma}
Suppose that we are given an RKHS of vector fields $V$ induced by a scalar diagonal reproducing kernel of the form $K(x,y)=k(x,y) I_{d}$. Given $v \in V$ of the form:
\begin{align*}
    v(x) = \sum_{i=1}^{n}{K(x,x_{i})\alpha_{i}},
\end{align*}
the $V$ norm (induced by $K$) has the following form:
\begin{align}\label{norm}
    \norm{v}_{V}^{2} = \sum_{i,j=1}^{n}{\alpha_{j}^{T}K(x_{i},x_{j})\alpha_{i}},
\end{align}
which for scalar diagonal kernels reduces to,
\begin{align}
    \norm{v}_{V}^{2} = \sum_{l=1}^{d}{\norm{v_{l}}_{V_{k}}^{2}},
\end{align}
where $v_{l}$ denotes the dimension $l$ component of $v$ and $V_{k}$ is the RKHS of real valued functions induced by $k$.
\end{lemma}
\begin{proof}
The identity \eqref{norm} is standard and follows from the following identities in RKHS
\begin{align*}
    \norm{v}_{V}^{2} = (Lv,v)\\
    L(K(\cdot,x)\alpha) = \alpha^{T}\delta_{x}.
\end{align*}
For the second claim, it is an easy computation that
    \begin{align*}
    \norm{v}_{V}^{2} = \sum_{i,j=1}^{n}{k(x_{i},x_{j})\alpha_{j}^{T}\alpha_{i}} = \sum_{i,j=1}^{n}{k(x_{i},x_{j})\sum_{l=1}^{d}{\alpha_{jl}\alpha_{il}}}
    = \sum_{i,j=1}^{n}{\sum_{l=1}^{d}{k(x_{i},x_{j})\alpha_{jl}\alpha_{il}}} \\ = \sum_{l=1}^{d}{\sum_{i,j=1}^{n}{{k(x_{i},x_{j})\alpha_{jl}\alpha_{il}}}} 
    = \sum_{l=1}^{d}{\norm{v_{l}}_{V_{k}}^{2}},
\end{align*}
where $v_{l}$ denotes the dimension $l$ component of $v$ so that
\begin{align*}
    v_{l} = \sum_{i=1}^{n}{k(\cdot,x_{i})\alpha_{il}},
\end{align*}
with weights $\alpha_{l}=(\alpha_{il})_{i=1}^{n}$.
\end{proof}

\section{MCMC for $m$-DPP}
A popular way to sample cheaply from an $m$-DPP is to run an MCMC chain that converges to the target $m$-DPP. There is a large literature on deriving fast mixing MCMC algorithms for $m$-DPP sampling. We give one example here from \cite{Anari}

\begin{algorithm}[H]\label[algorithm]{MCMKDPP}
\caption{MCMC for sampling an $m$-DPP}
\begin{algorithmic}[1]
\State Initialise number of points $m$, kernel function $k$, number of MCMC chain iterations $R$, $\bm{X}=\{x_{i}\}_{i=1}^{n}$, and uniformly sample an index set $S_{0}$ of size $m$.
\While{$r<R$}
    \State Sample $i$ uniformly from $S_{r}$ and $j$ uniformly from $\bm{X}-S_{r}$. Define the set $T=(S-{i})\cup {j}$.

    \State Compute transition probabilities $p_{ij}=\frac{1}{2}\min({1,\det(K_{T})/\det(K_{S_{r}}) })$  .
   \State Sample the next state so that with probability $p_{ij}$ we have $S_{r+1}=T$, otherwise $S_{r+1}=S_{r}$.
\EndWhile 
\State Return sample $S_{R}$
\end{algorithmic}
\end{algorithm}
Using a smart way to compute the inner loop, one can show the per-iteration cost is $\mathcal{O}(m^{2})$. This chain is observed to converge to the target $m$-DPP much faster in practice than the theoretical bounds, as evidenced in \cite{Anari}. Furthermore, one obtains the theoretical guarantees of \cref{CompressionofDiscreteMeasures}, asymptotically as the chain converges.

\bibliographystyle{siamplain}
\bibliography{references.bib}

\end{document}